\documentclass[reqno]{amsart}
\usepackage{amsmath,amssymb}
\usepackage{color}
\usepackage{enumerate}
\usepackage[nospace]{cite}

\usepackage{color} 
\definecolor{bleu_sombre}{rgb}{0,0,0.6}\definecolor{rouge_sombre}{rgb}{0.8,0,0}\definecolor{vert_sombre}{rgb}{0,0.6,0}
\usepackage[plainpages=false,colorlinks,linkcolor=bleu_sombre,
citecolor=rouge_sombre,urlcolor=vert_sombre,breaklinks]{hyperref}

\usepackage[margin={2.8cm,3.5cm}]{geometry}
\usepackage{placeins}

\usepackage{graphicx}
\graphicspath{ {images/} }

\newcommand{\real}{\mathbb{R}}

\newcommand{\intr}{\int_{\mathbb{R}}}

\newcommand{\N}{\mathbb{N}}
\newcommand{\R}{\mathbb{R}}
\newcommand{\C}{\mathbb{C}}

\newcommand{\ep}{\epsilon}
\newcommand{\ffi}{\varphi}

\newcommand{\lam}{\lambda}
\newcommand{\gam}{\gamma}
\newcommand{\al}{\alpha}
\newcommand{\om}{\omega}

\newcommand{\diff}{\,\mathrm{d}}
\newcommand{\dif}{\mathrm{d}}

\newcommand{\Hds}{H^{2,*}}
\newcommand{\DD}{\mathfrak{D}}

\newcommand{\norm}[1]{\ensuremath{\left\Vert #1 \right\Vert }}
\newcommand{\ps}[1]{\ensuremath{\left\langle #1 \right\rangle }}

\DeclareMathOperator \sgn{sgn}
\DeclareMathOperator \rge{rge}

\DeclareMathOperator \re{Re}

\newtheorem{theorem}{Theorem}
\newtheorem{lemma}[theorem]{Lemma}
\newtheorem{proposition}[theorem]{Proposition}
\newtheorem{corollary}[theorem]{Corollary}
\newtheorem{remark}[theorem]{Remark}

\numberwithin{equation}{section}
\numberwithin{theorem}{section}
 
\newtheorem{definition}[theorem]{Definition}
\makeatletter

\@addtoreset{equation}{section}  
\makeatother

\renewcommand{\leq}{\leqslant}	\renewcommand{\geq}{\geqslant}
\renewcommand\over[2]{{\,\buildrel #1\over#2\,}}

\newcommand{\inv}{^{-1}}

\newcommand {\limt}[2]{\xrightarrow[#1 \to #2]{}}

\newcommand{\abs}[1]{\left\vert #1\right\vert}        
\newcommand{\nr}[1]{\left\Vert #1\right\Vert}         
\newcommand{\innp}[2]{\left< #1 , #2 \right>}         

\newcommand{\Dom}{\Dc}			

\newcommand{\st}{\,:\,}					

\renewcommand{\Re}{\mathop{\rm{Re}}\nolimits}        
\renewcommand{\Im}{\mathop{\rm{Im}}\nolimits}        
\DeclareMathOperator{\Id}{Id}                        
\newcommand{\eqv}{\Longleftrightarrow}               

\renewcommand{\a}{\alpha}\renewcommand{\b}{\beta}\newcommand{\g}{\gamma}\newcommand{\G}{\Gamma}\renewcommand{\d}{\delta}
\newcommand{\z}{\zeta} \newcommand{\Th}{\Theta}\renewcommand{\k}{\kappa}\renewcommand{\l}{\lambda}\renewcommand{\L}{\Lambda}\newcommand{\m}{\mu}

\newcommand{\s}{\sigma}
\renewcommand{\t}{\tau}\newcommand{\f}{\varphi}\newcommand{\p}{\psi}\renewcommand{\o}{\omega}\renewcommand{\O}{\Omega}

\newcommand{\Ac}{{\mathcal A}}\newcommand{\Dc}{{\mathcal D}}\newcommand{\Jc}{{\mathcal J}}\newcommand{\Lc}{{\mathcal L}}
\newcommand{\qandq}{\quad \text{and} \quad}

\usepackage{mathrsfs}

\usepackage{accents}
\makeatletter
\newcommand{\sbullet}{%
  \hbox{\fontfamily{lmr}\fontsize{.4\dimexpr(\f@size pt)}{0}\selectfont\textbullet}}
\DeclareRobustCommand{\mathbullet}{\accentset{\sbullet}}
\makeatother

\begin{document}

\title[a nonlinear Klein--Gordon equation with delta potentials]
{Stability of standing waves for a nonlinear Klein--Gordon equation with delta potentials}

\author[E. Csobo]{Elek Csobo}
\address{Delft University of Technology\\
Van Mourik Broekmanweg~6\\
2628 XE Delft, The Netherlands}
\email{e.csobo@tudelft.nl}

\author[F. Genoud]{Fran\c cois Genoud}
\address{Ecole Polytechnique F\'ed\'erale de Lausanne \\
EPFL Station 4 \\
1015 Lausanne, Switzerland}
\email{francois.genoud@epfl.ch}

\author[M. Ohta]{Masahito Ohta}
\address{Tokyo University of Science \\
1-3 Kagurazaka, Shinjukuku \\
Tokyo 162-8601, Japan}
\email{mohta@rs.tus.ac.jp}

\author[J. Royer]{Julien Royer}
\address{Universit\'e Toulouse 3 \\
118 route de Narbonne \\
31062 Toulouse Cedex~9, France}
\email{julien.royer@math.univ-toulouse.fr}

\keywords{Nonlinear Klein--Gordon equation, standing waves, orbital stability, delta potential}

\begin{abstract}

In this paper, we study local well-posedness and orbital stability of standing waves for a  
singularly perturbed one-dimensional nonlinear Klein--Gordon equation.
We first establish local well-posedness of the Cauchy problem by a fixed point argument. 
Unlike the unperturbed case, 
a noteworthy difficulty here arises from the possible non-unitarity of the 
semigroup generating the corresponding linear evolution. 
We then show that the equation is Hamiltonian and we establish several stability/instability results 
for its standing waves. Our analysis relies on a detailed study of the spectral properties 
of the linearization of the equation, and on the well-known `slope condition' for orbital stability.
\end{abstract}

\maketitle

\newcommand{\stepp}{\noindent {\bf $\bullet$}\quad }

\newcommand{\EE}{\mathscr E}\newcommand{\EEg}{\mathscr E}
\newcommand{\HH}{\mathscr H}\newcommand{\HHg}{{\mathscr H_\g}} \newcommand{\tHH}{{\tilde {\mathscr H}}}
\newcommand{\dHH}{{\mathbullet {\mathscr H}}} \newcommand{\dHHg}{{\mathring {\mathscr H}_\g}}
\newcommand{\LL}{\mathscr L}
\newcommand{\XX}{\mathscr X}
\newcommand{\Aga}{\Ac_{\g,\a}}\newcommand{\Aoo}{\Ac_{0,0}}

\newcommand{\Lp}{{L}^+_{\b}}\newcommand{\Lpo}{{L}^+_{0}}
\newcommand{\Lm}{{L}^-_{\b}}
\newcommand{\Lcb}{\Lc_{\b}}
\newcommand{\Lob}{\Lc_{\b}}
\newcommand{\Lpp}{L_\b''}
\newcommand{\RLpp}{\tilde L_\b''}
\newcommand{\tildeL}{\tilde T}

\newcommand{\ess}{{\mathrm{ess}}}

\section{Introduction}

The purpose of this work is to initiate the study of the Cauchy problem for 
a singularly perturbed one-dimensional nonlinear Klein--Gordon equation, namely†
\begin{equation} \label{eq-wave-0}
\begin{cases}
u_{tt} - u_{xx} + m^2 u + \g \d u + i \a \d u_t -|u|^{p-1}u = 0,\\
u(t,x) \limt {\abs x} \infty 0,\\
(u(t),\partial_t u(t))|_{t = 0} = (u_0,u_1),
\end{cases}
\end{equation}
where $u:\R\times\R\rightarrow \C$, $m>0$, $\a, \g \in \R$ are parameters and 
$p>1$ determines the strength of the nonlinearity.
The two coefficients $\delta=\delta(x)$ are singular perturbations both given by a Dirac mass at $x=0$,
often referred to as a `delta potential' in the context of one-dimensional evolution equations.
Such space-dependent problems are sometimes termed `inhomogeneous', as they model 
wave propagation in inhomogeneous media.

The condition that $u$ vanishes at spatial infinity reflects a common physical requirement of having
spatially localized waves, sometimes called `solitons'. We shall in fact seek solutions of 
\eqref{eq-wave-0} in $H^1(\R)$.

We will show that the evolution generated by \eqref{eq-wave-0} admits a peculiar Hamiltonian
formalism, with a symplectic structure depending on the coupling constant 
$\a\in\R$. In fact, if $\a\in\C\setminus\R$, the Hamiltonian of the system is not constant along the flow,
as can be deduced from the symplectic structure in Section~\ref{ham.sec}, or by a direct calculation
using a smooth solution.
Furthermore,
this Hamiltonian system is phase invariant (if $u$ is a solution, so is $e^{i\theta}u$, for any
$\theta\in\R$), and thus possesses standing wave solutions, of the form $u(t,x)=e^{i \o t}\ffi(x)$,
with $\o\in\R$ and $\ffi:\R\to\R$.
The stability of standing waves in Hamiltonian systems with symmetries 
has attracted a lot of attention since the 1980's. 
So far, this issue has been fairly well understood in homogeneous media, 
while in inhomogeneous media it is still a subject of intense research, 
both theoretically and experimentally. 
Inhomogeneous nonlinear dispersive equations appear in various fields of physics 
such as nonlinear optics, cold quantum gases (e.g.~Bose--Einstein condensates), 
plasma physics, etc. 
More specifically,
our interest in the present problem was initially motivated by \cite{Danshita}, where
\eqref{eq-wave-0} arises as an effective model for a superfluid Bose gas.

The nonlinear Klein--Gordon equation in homogeneous media has been extensively studied. 
A detailed presentation of the local and global well-posedness theory can be found in \cite{CH98}. 
Orbital stability of standing wave solutions was first addressed in the classical papers of 
Shatah \cite{Shatah1, Shatah2}, and Shatah and Strauss \cite{Shatah3}. 
They proved that, in $N$ space dimensions, standing waves of any frequency are orbitally 
unstable if $p\geq 1+4/N$. If $1<p<1+4/N$, then there exists a critical frequency $\om_c$
such that a standing wave of frequency $\om$ is orbitally stable if $\om_c<|\om|<m$ 
and unstable if $|\om|<\om_c$. Strong instability by blow-up in finite time was 
studied by Liu, Ohta and Todorova \cite{Liu}, and by Ohta and Todorova \cite{Ohta1, Ohta2}. 
In \cite{Jeanjean}, Jeanjean and Le~Coz introduced a mountain-pass approach 
to orbital stability for the Klein--Gordon equation, which allowed them to simplify the classical 
proofs and to obtain new results. 
In \cite{Bellazzini}, Bellazzini, Ghimenti and Le~Coz proved the existence of 
multi-solitary waves for the nonlinear Klein--Gordon equation.

The effect of a singular potential on the dynamics of the nonlinear Schr\"odinger equation has 
recently attracted substantial attention. Well-posedness of the Cauchy problem in the presence 
of a delta potential was studied in \cite{Adami}, while scattering
was addressed in \cite{ikeda,datchev,Goodman}. The stability of 
standing waves of the nonlinear Schr\"odinger equation with a delta potential was studied in 
\cite{Fukuizumi1, Fukuizumi2,LeCoz, Genoud, masaki} in various regimes. 
(Some authors consider $\delta'$ interactions as well, see e.g.~\cite{Adami2,pava}.)
Stability properties of so-called `black solitons' (standing waves with $|\ffi(x)|\to 1$ as $|x|\to\infty$) 
were also recently addressed in \cite{black}.

The present work is a first step in the study of the nonlinear Klein--Gordon equation with delta potentials. 
Our main goal here is to discuss orbital stability of standing waves of \eqref{eq-wave-0}. 
Shortly after the seminal works \cite{Shatah1,Shatah2,Shatah3},
a general theory of orbital stability for infinite-dimensional Hamiltonian systems with symmetries 
was established by Grillakis, Shatah and Strauss \cite{GSS}. Their approach, based
on the so-called `energy-momentum' method of geometric mechanics, 
was recently revisited by De~Bi\`evre, Rota Nodari and the second author \cite{DeGeRo15}, 
and by Stuart \cite{Stuart}. 
Under general assumptions on the dynamical system, conditions are given in these papers 
for orbital stability and instability. Of course, in order to discuss stability of standing waves,
an essential preliminary step is to prove that 
the Hamiltonian system under consideration is locally well-posed.
We shall thus start by addressing this issue, which is far from obvious 
in the context of \eqref{eq-wave-0}.

The singular terms in \eqref{eq-wave-0} should be interpreted in the sense of distributions. 
Let us assume that $u$ and $u_t$ are continuous at $x=0$. 
$\delta u$ is then defined by $\ps{\delta u,w}=\Re u(0)\bar{w}(0)$,
for any function $w$ continuous at $x=0$. And $\delta u_t$ is defined similarly.
Now, solutions of the equation in \eqref{eq-wave-0}
will be continuous functions satisfying the corresponding unperturbed equation 
(with $\gamma=\alpha=0$) pointwise, outside of $x=0$, together with the jump condition
\begin{equation}\label{jump}
u'(0^+)-u'(0^-)=\g u(0) + i\a u_t(0).
\end{equation}
Formally, this relation is indeed obtained from the equation with the delta potentials
by integrating it over $x\in(-\ep,+\ep)$ and letting $\ep\to0^+$. The notion of solution will be
made more precise in Section~\ref{cauchy.sec}, once the appropriate functional setting has been introduced.

Although writing the delta potentials explicitly may be useful for some formal calculations, 
we now introduce a functional-analytic formulation, based on the jump condition \eqref{jump},
which will make our analysis more transparent.
It is convenient to reformulate the initial-value problem \eqref{eq-wave-0}
as a first order system for the dependent variables $(u,v)=(u,u_t)$. 
We will seek solutions to \eqref{eq-wave-0} with $(u,v)\in\HH=H^1(\R)\times L^2(\R)$, 
which we regard as a real Hilbert space, endowed with the inner product
\[
\ps{(u_1,v_1),(u_2,v_2)}_{\HH}=\ps{u_1',u_2'}_{L^2}+\ps{u_1,u_2}_{L^2}+\ps{v_1,v_2}_{L^2},
\]
where the real $L^2$ inner product is defined as
\[
\ps{u,v}_{L^2}=\Re\intr u\bar v \diff x.
\]
Here and henceforth, $'$ denotes differentiation with respect to $x\in\R$. 

We identify $L^2(\R) \times L^2(\R)$ with its dual. 
Then the dual $\HH^*$ of $\HH$ is $H^{-1}(\R)\times L^2(\R)$, and 
for any $(\ffi,\psi)\in L^2(\R)\times L^2(\R) \subset H^{-1}(\R)\times L^2(\R)$, the duality pairing
is given by
\[
\ps{(\ffi,\psi),(u,v)}_{\HH^*\times\HH}=\ps{\ffi,u}_{L^2}+\ps{\psi,v}_{L^2}, \quad (u,v)\in\HH.
\]
We shall merely write $\ps{\cdot,\cdot}$ for $\ps{\cdot,\cdot}_{\HH^*\times\HH}$ 
when no confusion is possible.

The central object in our discussion of the well-posedness of \eqref{eq-wave-0} in 
Section~\ref{cauchy.sec} is the generator $\Ac$ of the corresponding linear evolution, defined as 
\begin{equation} \label{def-Ac}
\Ac =
\begin{pmatrix}
0 & \Id_{L^2} \\
\partial_x^2 - m^2 & 0
\end{pmatrix},
\end{equation}
with domain
\begin{equation} \label{def-DD}
\DD = \left\{ (u,v) \in \Hds \times  H^1 : u'(0^+) - u'(0^-) = \g u(0) + i \a v(0) \right\} \subset \HH,
\end{equation}
where
\[
\Hds = H^1(\R) \cap  H^2(\R \setminus \{ 0\}).
\]
Note that the effect of the delta potentials is encoded in the domain of the generator.

We will show that the operator $\Ac$ generates a $C^0$-semigroup on $\HH$ which,
remarkably, may not be a unitary group. In contrast to the 
classical unperturbed case, it is in general only exponentially bounded. We recall that a skew-adjoint operator on a Hilbert space generates a unitary group (the results which we use in this paper can be found in \cite{CH98}, we refer for instance to \cite{engel} for a more detailed presentation of the theory of semigroups). Skew-adjointness of $\Ac$ depends on the Hilbert structure chosen on $\HH$. For $\gamma$ non-negative (or negative and close to 0), it will be possible to define on $\HH$ an inner product whose corresponding norm is equivalent to the usual one and which makes $\Ac$ skew-adjoint. However, for other values of $\g$, we can only consider on $\HH$ a Hilbert structure for which $\Ac$ generates a continuous (possibly exponentially growing) semigroup. This is in fact enough for our purpose. Notice that the parameter $\a$ will not play any role in this discussion.

Using Duhamel's formula and the Banach Fixed Point Theorem, 
we will then construct, for any initial data in $\HH$, a unique local in time solution. 
We also prove the blow-up alternative and continuous dependence on the initial data
for this solution.

Next, standing waves, which will be our main focus, are solutions of \eqref{eq-wave-0} of the form
\[
u_\omega(t,x)=e^{i \o t}\ffi_\omega(x),
\]
where $\omega\in\R$, and $\ffi_\omega\in \Hds$ is real-valued and 
satisfies the stationary equation
\begin{equation}\label{stat}
-\ffi''+(m^2-\omega^2)\ffi+(\gamma-\alpha\omega)\delta\ffi-|\ffi|^{p-1}\ffi=0,
\end{equation}
which will be interpreted as
\begin{equation}\label{stat2}
-\ffi''+(m^2-\omega^2)\ffi-|\ffi|^{p-1}\ffi=0, \quad \text{a.e. } x\in\real,
\end{equation}
together with the jump condition
\begin{equation}\label{jump2}
\ffi'(0^+)-\ffi'(0^-)=(\g -\a\om)\ffi(0).
\end{equation}
Non-trivial localized solutions to this problem exist if and only if
\begin{equation} \label{omega_range}
m^2-\om^2>\frac{(\g-\a\om)^2}{4},
\end{equation}
in which case they are given by the explicit formula 
(see Proposition~1 and Remark~1 in \cite{LeCoz})
\begin{equation} \label{explicit}
\ffi_{\omega}(x)=
\left[\frac{(p+1)(m^2-\omega^2)}{2}\textnormal{sech}^2
\left(\frac{(p-1)\sqrt{m^2-\omega^2}}{2}|x|+\tanh^{-1}
\left(-\frac{\gamma-\alpha\omega}{2\sqrt{m^2-\omega^2}}\right)\right)\right]^{\frac{1}{p-1}}.
\end {equation}
In particular, there are no standing wave solutions of \eqref{eq-wave-0} when $m=0$.

\begin{definition} 
\rm
For any fixed $m$, $\a$ and $\g$, we shall say that $\om$ is {\em admissible} 
if it satisfies the relation \eqref{omega_range}. 
\end{definition}

In order to reveal the Hamiltonian structure of the initial-value problem \eqref{eq-wave-0}, 
we shall follow the notation and terminology of \cite{DeGeRo15}.
The Hamiltonian {\em energy} functional associated with \eqref{eq-wave-0}
is given by
\begin{equation}\label{energy}
E(u,v)=\frac12\nr{u'}_{L^2}^2+\frac{m^2}{2}\nr{u}_{L^2}^2
+\frac12\nr{v}_{L^2}^2+\frac{\gamma}{2}|u(0)|^2
-\frac{1}{p+1}\intr |u|^{p+1} \diff x.
\end{equation}
We shall prove in Section~\ref{ham.sec} that $E$ is a constant of the motion. 
Another important quantity is conserved along the flow of the solution, 
namely the {\em charge}, defined as
\begin{equation}\label{charge}
Q(u,v)=\Im\intr u\bar v \diff x -\frac{\alpha}{2}|u(0)|^2.
\end{equation}
We will establish in Section~\ref{ham.sec} that $E, Q \in C^2(\HH,\R)$.
Let us now introduce the {\em symplector} $\Jc:\HH\rightarrow \HH^*$ defined by
\[
\Jc(u,v)=(-i\a\d u-v,u).
\]
This notion, somewhat more flexible than that of a symplectic map, is
introduced in \cite[Sec.~6]{DeGeRo15} to define Hamiltonian systems. It is noteworthy
that the coupling constant $\a$ appears here in the symplectic structure itself.
In this framework, the equation in \eqref{eq-wave-0} is formulated as
the Hamiltonian system
\begin{equation}\label{ham}
\Jc\frac{\dif}{\dif t}U(t)=E'(U(t)),
\end{equation}
where $E'$ denotes the Fr\'echet derivative of $E$. 
A standing wave is now a solution of the form
\begin{equation}\label{stwave}
U_\omega(t,x)=e^{i{\om} t}\Phi_\omega(x),
\end{equation}
where $\Phi_\omega:=(\ffi_\omega,i\om\ffi_\omega)$
satisfies the stationary equation
\begin{equation}\label{ham_stat-0}
E'(\Phi_\omega)+\om Q'(\Phi_\omega)=0.
\end{equation}

We will study the orbital stability of the standing waves \eqref{stwave}, for admissible values
of $\omega\in\R$,
with respect to the symmetry group $\mathbb{S}^1$ acting on $\HH$ through
\begin{equation}\label{action}
T(\theta)(u,v)=e^{i\theta}(u,v), \quad \theta\in\R.
\end{equation}
This group action leaves \eqref{ham} invariant. 
The corresponding notion of orbital stability is the following.
\begin{definition} 
\rm
For a fixed $\omega_0\in\R$, the standing wave $e^{i{\om_0} t}\Phi_{\om_0}$
is {\em orbitally stable} if the following holds: 
for any $\ep>0$ there is a $\d>0$ such that,
if $U(t)$ is a solution of \eqref{ham-0}, then we have
\begin{equation}\label{stability}
\norm{U(0)-\Phi_{\omega_0}}_\HH<\d \implies 
\inf_{\theta\in\real}\norm{U(t)-e^{i\theta}\Phi_{\omega_0}}_\HH<\ep \quad\text{for all} \ t\in\real.
\end{equation}
Otherwise, $\Phi_{\om_0}$ is said to be {\em orbitally unstable}.
\end{definition}

In addition to orbital stability, we will also prove some linear instability results.
Writing a solution $U$ of \eqref{ham} in the form $U(t)=e^{i{\om_0} t}(\Phi_{\om_0}+V(t))$, 
we have that, at first order, $V$ satisfies the linearized equation
\begin{equation} \label{lin}
\Jc \frac{\diff }{\diff t} V(t)= L_{\om_0}'' (\Phi_{\omega_0})V(t),
\end{equation}
where $L_\om$ is defined in \eqref{lyap}.

\begin{definition} 
\rm
The standing wave $e^{i{\om_0} t}\Phi_{\om_0}$ is {\em linearly unstable} if 
$0$ is a linearly unstable solution (in the sense of Lyapunov) of \eqref{lin}.
\end{definition}

In Section~\ref{spectrum.sec}, we will carry out a stability analysis based on
the energy-momentum method developed in \cite{GSS,DeGeRo15,Stuart}. More precisely,
our proofs will make use of the well-known {\em slope condition} (also known as the
`Vahkitov--Kolokolov criterion'), which states that the standing wave $\Phi_{\omega_0}$
is stable/unstable provided
\begin{equation} \label{slope}
\frac{\diff}{\diff \om}\Big\vert_{\om=\om_0} Q(\Phi_{\om})>0\quad \Big/ \quad \frac{\diff}{\diff \om}\Big\vert_{\om=\om_0} Q(\Phi_{\om})<0,
\end{equation}
where the charge of the standing wave \eqref{stwave} is explicitly given by
\begin{equation}
Q(\Phi_{\om})=Q(\ffi_\om,i\om\ffi_\om)=-\om\norm{\ffi_\om}_{L^2}^2-\frac{\al}{2}|\ffi_\om(0)|^2.
\end{equation}

The stability/instability of $\Phi_\omega$ in fact relies on a subtle combination 
of the slope condition \eqref{slope}
and suitable spectral properties of the linearization of \eqref{ham} 
(see e.g.~\cite[Sec.~10.3]{DeGeRo15} for a detailed discussion in the context of the 
nonlinear Schr\"odinger equation). The spectral conditions are conveniently expressed in terms of
the Lyapunov functional $L_\om:\HH\to\real$ associated with \eqref{ham},
defined by
\begin{equation}\label{lyap}
L_\om(u,v)=E(u,v)+\om Q(u,v).
\end{equation}

Let $\tilde{R}=\text{diag}(R,\mathrm{Id}_{L^2}):\HH\to\HH^*$, 
where $R=-\partial_x^2+1:H^1(\R)\to H^{-1}(\R)$ is the Riesz isomorphism. 
It follows from the results of Sections \ref{ham.sec} and \ref{spectrum.sec} that, for any $\omega\in\R$, 
$L_\omega\in C^2(\HH,\R)$, and that $\tilde{R}^{-1}L_\om''(\Phi_\omega):\HH\to\HH$
is a bounded selfadjoint operator. Let us denote by 
$\sigma(\tilde R^{-1} L_\om''(\Phi_\omega))\subset\R$ its spectrum.
The relevant spectral conditions for stability 
are then formulated as follows.

\begin{itemize}

\item[(S1)]
There exists  $\lambda_\om\in\R$ such that 
$\sigma(\tilde R^{-1} L_\om''(\Phi_\omega))\cap(-\infty,0)=\{-\lambda_\om^2\}$ 
and the subspace $\ker(L_\om''(\Phi_\omega)+\lambda_\om^2\tilde{R})$ is one-dimensional.

\item[(S$1'$)] $\tilde R^{-1} L_\om''(\Phi_\omega)$ has two negative eigenvalues 
(counted with multiplicities): either there exist $\lambda_\om,\mu_\om\in\R$ such that 
$\sigma(\tilde R^{-1} L_\om''(\Phi_\omega))\cap(-\infty,0)=\{-\lambda_\om^2,-\mu_\om^2\}$, 
$\lambda_\o^2 \neq \mu_\o^2$, 
and the subspaces $\ker(L_\om''(\Phi_\omega)+\lambda_\om^2\tilde{R})$ and
$\ker(L_\om''(\Phi_\omega)+\mu_\om^2\tilde{R})$
are both one-dimensional; or there exists $\lambda_\om\in\R$ such that 
$\sigma(\tilde R^{-1} L_\om''(\Phi_\omega))\cap(-\infty,0)=\{-\lambda_\om^2\}$ 
and the subspace $\ker(L_\om''(\Phi_\omega)+\lambda_\om^2\tilde{R})$ has dimension 2.

\item[(S2)] $\ker L_\om''(\Phi_\omega)=\mathrm{span}\{i \Phi_\omega\}$.

\item[(S3)] Apart from the non-positive eigenvalues, $\sigma(\tilde R^{-1} L_\om''(\Phi_\omega))$ 
is positive and bounded away from zero.
\end{itemize}

In the present context, the Cauchy problem \eqref{eq-wave-0} being locally well posed,
the main results of \cite{GSS,DeGeRo15,Stuart} imply that, if the standing wave $\Phi_{\omega_0}$
satisfies (S1)--(S3), then it is orbitally stable/unstable provided \eqref{slope} holds. 
In case (S1) is replaced by (S$1'$), we will discuss linear instability of the standing waves,
by means of results obtained in \cite{Gs2}. 
We shall therefore carry out a thorough spectral analysis to see when conditions (S1)--(S3)
(resp.~(S$1'$)--(S3)) are satisfied, depending on the values of the parameters. By
discussing the slope condition for some values of the parameters, we will then prove 
various stability/instability results in $\HH$ and in the subspace $\HH_{\mathrm{rad}}$
of radial functions.

In this analysis, we shall benefit from the explicit dependence of the solution on the parameters, but the 
calculations required for the slope condition are rather involved. 
This difficulty is reflected in the intricate
form of the results we present in Section~\ref{spectrum.sec} 
and explains why we decided to focus on some regimes and refrained from
attempting a comprehensive analysis. 
Of course, numerics might come in handy to 
discuss the slope condition outside the scope of our analytical results. We conclude this introduction
with the following table, which captures simply what ought to be checked in order to obtain
stability/instability results. The integer $n_\omega$ (resp.~$n_{\omega,\mathrm{rad}}$) 
denotes the number of negative eigenvalues (counted with multiplicities) of the operator 
$\tilde R^{-1} L_\om''(\Phi_\omega)$ in $\HH$ (resp.~$\HH_{\mathrm{rad}}$).\footnote{It 
follows from Proposition~\ref{two_negative} and Remark~\ref{rad}
that $n_{\omega,\mathrm{rad}}=1$ whenever $n_\omega=2$.}

\[
\begin{array}{|c|c|c|}
\hline
& n_\omega = 1 & n_\omega= 2 \quad \text{and} \quad  n_{\omega,\mathrm{rad}}= 1 \\
\hline
\frac {\diff}{\diff \o} Q(\Phi_\o) > 0 &  \text{orbitally stable} & \text{linearly unstable} \\
\hline
\frac {\diff}{\diff \o} Q(\Phi_\o) < 0 & \text{orbitally unstable} & \text{orbitally unstable in $\HH_{\mathrm{rad}}$, 
hence in $\HH$} \\ \hline
\end{array}
\]


\section{Local well-posedness of the Cauchy problem}\label{cauchy.sec}
In this section we discuss the local well-posedness of the Cauchy problem \eqref{eq-wave-0}. 
In order to apply the standard theory of operator semigroups, we reformulate \eqref{eq-wave-0} 
as a first order system on $\HH$. We consider on $\HH$ the operator $\Ac$ 
defined by \eqref{def-Ac}--\eqref{def-DD}. Given $f : \R \to \R$ we set, for $U = (u,v) \in \HH$,
\[
F(U) = \begin{pmatrix} 0 \\ f(u) \end{pmatrix}.
\]
With $f(u)=\abs{u}^{p-1}u$ and $U_0 = (u_0,u_1)$, \eqref{eq-wave-0} can be rewritten as
\begin{equation} \label{eq-wave-Ac}
\begin{cases}
U_t(t) - \Ac U(t) = F(U(t)),\\
U(0) = U_0.
\end{cases}
\end{equation}
We will show that $\Ac$ generates a strongly continuous semigroup on $\HH$, which will allow us to establish the local well-posedness of \eqref{eq-wave-0}.

\begin{definition}
\rm
Let $T\in(0,\infty]$.
\begin{itemize}
\item A strong solution to \eqref{eq-wave-Ac} is a function $U\in C^0([0,T),\DD)\cap C^1([0,T),\HH)$ such that \eqref{eq-wave-Ac} holds on $[0,T)$. We say that $u$ is a strong solution of \eqref{eq-wave-0} on $[0,T)$ if $(u,u_t)$ is a strong solution of \eqref{eq-wave-Ac}.
\item A weak solution of \eqref{eq-wave-Ac} is a function $U\in C^0([0,T), \HH)$ such that, for all $t\in[0,T)$, there holds
\begin{equation} \label{eq-Duhamel}
U(t) = e^{t\Ac} U_0 + \int_0^t e^{(t-s)\Ac} F(U(s)) \, \diff s.
\end{equation}
We say that $u$ is a weak solution of \eqref{eq-wave-0} on $[0,T)$ if $(u,u_t)$ is a weak solution of \eqref{eq-wave-0}.
\end{itemize}
\end{definition}

We begin with a lemma which ensures, in particular, that $\Ac$ is densely defined.

\begin{lemma} \label{lem-DD-dense}
$\DD$ is dense in $\HH$.
\end{lemma}

\begin{proof}
Let $(u,v) \in \HH$. We can consider a sequence $(v_n)$ in $H^1$ such that $v_n(0) = 0$ 
and $v_n \to v$ in $L^2$. We then choose a sequence $(u_n)$ in $H^2$ 
which converges to $u$ in $H^1$. For $n \in \N$ and $x \in \R$, we set 
\[
\z_n(x) = 1 + \frac {\g \abs x} 2 e^{-nx^2}.
\]
We have $\z_n'(0^+) = - \z_n'(0^-) = \frac \g 2$, so 
\[
(u_n \z_n)'(0^+) - (u_n \z_n)'(0^-) = u_n(0) \big( \z_n'(0^+) - \z_n'(0^-) \big) = \g (u_n\z_n)(0).
\]
This proves that $(u_n \z_n, v_n)$ belongs to $\DD$ for all $n \in \N$. Moreover, 
\[
\nr{\z_n -1}_{L^\infty} \limt n {+\infty} 0 \quad \text{and} \quad \nr{\z_n'}_{L^\infty} \limt n {+\infty} 0,
\]
so $\nr{u_n \z_n - u_n}_{H^1} \to 0$, and hence $(\z_n u_n,v_n)$ goes to $(u,v)$ in $\HH$.
\end{proof}


\subsection{Linear evolution in the energy space}
In this subsection we show that the operator $\Ac$ generates a strongly continuous group on $\HH$. We know that if $\Ac$ is skew-adjoint then it generates a one parameter unitary group on $\HH$. Since the notion of skew-adjointness depends on the inner product, we first discuss the choice of a suitable Hilbert structure on $\HH$.

For $\m \geq 0$ we introduce on $\HH$ the quadratic form defined by
\begin{equation} \label{norm-HH-gamma}
\Vert (u,v)\Vert_{\HH, \mu, \gam}^2= \Vert u'\Vert^2_{L^2}+\mu^2\Vert u\Vert^2_{L^2}+\gam|u(0)|^2+\Vert v\Vert^2_{L^2}.
\end{equation}
We denote by $\langle\cdot,\cdot\rangle_{\HH, \mu, \gam}$ the corresponding bilinear form. With $\m = m$ we observe that, for $U = (u,v) \in \DD$,
\begin{equation} \label{eq-AU-U}
\begin{aligned}
\langle \Ac U, U \rangle_{\HH,m,\g}
&=\innp{v'}{u'}_{L^2} + m^2 \innp v u _{L^2} + \g \Re( v(0) \overline u(0) ) + \innp{u''}{v}_{L^2} - m^2 \innp u v _{L^2} \\
&=\innp{v'}{u'}_{L^2} + \g \Re ( v(0) \overline u(0))- \innp{u'}{v'}_{L^2}-\Re( \gamma u(0) \overline v(0) + i\a|v(0)|^2 )\\
&=0.
\end{aligned}
\end{equation}
This makes $\langle\cdot,\cdot\rangle_{\HH, m, \gam}$ a good candidate to be a suitable inner product on $\HH$. However, for negative $\gamma$, it may happen that the corresponding quadratic form takes negative values. In this case we have to choose a larger parameter $\mu$.

\begin{lemma} \label{lem-norme-H1}
Let
\[
\mu_0 = 
\begin{cases}
0 & \text{if } \gamma \geq 0,\\
\frac {\abs{\gamma}} 2 & \text{is } \gamma < 0.
\end{cases}
\]
Then for $\mu > \mu_0$ there exists $C_\mu \geq 1$ such that, for all $u \in H^1$, we have 
\begin{equation} \label{equiv-H1}
C_\mu^{-1} \nr{u}_{H^1}^2 \leq \nr{u'}_{L^2}^2 + \mu^2 \nr{u}_{L^2}^2 
+ \gamma \abs{u(0)}^2 \leq C_\m \nr{u}_{H^1}^2. 
\end{equation}
In particular, the functional $\Vert \cdot \Vert_{\HH, \mu, \gam}$ is a norm on $\HH$, 
equivalent to the usual one.
\end{lemma}

\begin{proof}
For $u\in H^1$ we have
\begin{equation}\label{trace}
\abs \gamma |u(0)|^2 = 2 \abs \gamma \Re\int^0_{-\infty} u(x)\bar{u}'(x)\diff x 
\leq 2 \abs \gamma \Vert u\Vert_{L^2}\Vert u'\Vert_{L^2}
\leq  \frac {\nr{u'}_{L^2}^2} 2 + 2 \gamma ^2 \nr{u}_{L^2}^2.
\end{equation}
This gives in particular the second inequality of \eqref{equiv-H1}. 
By Theorem I.3.1.4 in \cite{Albeverio}, we have
\[
\nr{u'}_{L^2}^2 + \gamma \abs{u(0)}^2 + \mu_0^2 \nr{u}_{L^2}^2 \geq 0
\]
for all $u \in H^1$. Then for $\epsilon > 0$ we have 
\begin{align*}
\nr{u'}_{L^2}^2 + \mu^2 \nr{u}_{L^2}^2 + \gamma \abs{u(0)}^2 
& \geq 2 \epsilon \nr{u'}_{L^2}^2 + 2 \epsilon \gamma \abs{u(0)}^2 + (\m^2- (1-2\epsilon) \m_0^2) \nr{u}_{L^2}^2\\
& \geq \epsilon \nr{u'}_{L^2}^2 + \left( \mu^2 - (1-2\epsilon) \m_0^2 - 4 \epsilon \gamma^2 \right) \nr{u}_{L^2}^2.
\end{align*}
With $\epsilon > 0$ small enough, this gives the first inequality in \eqref{equiv-H1}, and the second statement of the proposition follows.
\end{proof}We intend to prove the following proposition.

\begin{proposition} \label{prop-semigroups}
The operator $\Ac$ generates a $C^0$-semigroup on $\HH$. Moreover, there exist $M \geq 0$ 
and $\b \geq 0$ such that, for all $t \in \R$, we have 
\[
\nr{e^{t\Ac}}_{\Lc(\HH)} \leq M e^{\b |t|}.
\]
\end{proposition}

For this we need the following lemma.

\begin{lemma} \label{lem-res-u}
Let $\mu\geq0$ be as in Lemma \ref{lem-norme-H1} and $\lam\geq\sqrt{\mu^2-m^2}$.
\begin{enumerate}[\rm (i)]
\item The bounded operator
\begin{align} \label{op-L2-lambda}
-\partial_{xx}+(m^2+\l^2)+(\g +i\l\a )\d : H^1 \to H\inv
\end{align}
has a bounded inverse, which we denote by $R(\l)$.

\item Let $\f \in L^2$ and $\p \in H^1$. Then $R (\l) (\f - \d \p)$ belongs to $\Hds$. 
It is the unique solution $u$ in $\Hds$ of the problem 
\begin{equation} \label{eq-Helmholtz}
\begin{cases}
-u'' + (m^2 + \l^2) u = \f,\\
u'(0^+) - u'(0^-) = (\g+i\l\a) u(0) + \p(0).
\end{cases}
\end{equation}
\end{enumerate}
\end{lemma}

\begin{proof}
By Lemma~\ref{lem-norme-H1} we have, for all $u\in H^1$,
\begin{align*}
\langle (-&\partial_{xx}+(m^2+\l^2)+(\g+i\l\a)\d)u,u \rangle_{H\inv, H^1}
\geq \|u'\|^2_{L^2}+\mu^2\|u\|^2_{L^2}+\g|u(0)|^2\gtrsim \nr{u}_{H^1}^2.
\end{align*}
Similarly, for all $u$ and $v$ in $H^1$,
\[
\langle (-\partial_{xx}+(m^2+\l^2)+(\g+i\l\a)\d)u,v \rangle_{H\inv, H^1}\lesssim \nr{u}_{H^1}\nr{v}_{H^1}.
\]
Hence, the operator \eqref{op-L2-lambda} has a bounded inverse by the Lax--Milgram Lemma. 

Let us assume that \eqref{eq-Helmholtz} has a solution $u \in\Hds$. For all $w\in H^1$ we have
\[
\langle -u'',w \rangle_{L^2} +(m^2+\l^2)\langle u,w \rangle_{L^2}=\langle \f,w \rangle_{L^2}.
\]
Integrating by parts and using the jump condition \eqref{eq-Helmholtz}, we get
\begin{equation} \label{eq-H1-Hinv}
\langle u', w'\rangle_{L^2} +(m^2+\l^2)\langle u,w \rangle_{L^2} +\Re((\g+i\l\a)u(0)\bar{w}(0))
=\langle \f,w\rangle_{L^2}-\Re(\p(0)\bar{w}(0)).
\end{equation}
This proves that $u = R(\l) (\f - \d \p)$. Conversely, let $u = R(\l) (\f - \d \p) \in H^1$. 
Then \eqref{eq-H1-Hinv} holds for all $w \in C_0^\infty(\R \setminus \{0\})$, 
so $u$ belongs to $\Hds$ and $-u'' + (m^2+\l^2) u = \f$. 
We now write \eqref{eq-H1-Hinv} with $w \in C_0^\infty(\R)$ such that $w(0) = 1$, 
which yields the jump condition in \eqref{eq-Helmholtz}. 
\end{proof}

We can now prove Proposition \ref{prop-semigroups}.

\begin{proof}[Proof of Proposition \ref{prop-semigroups}]
Consider $\m$ as given by Lemma \ref{lem-norme-H1}. For $U = (u,v) \in \DD$, we have
\begin{align*}
\langle \Ac U, U \rangle_{\HH,\m,\g}
&=\innp{v'}{u'}_{L^2} + \m^2 \innp v u _{L^2} + \g \Re( v(0) \overline u(0) ) 
+ \innp{u''}{v}_{L^2} - m^2 \innp u v _{L^2} \\
&=\innp{v'}{u'}_{L^2} + (\m^2-m^2) \innp v u _{L^2}
+ \g \Re ( v(0) \overline u(0))-\innp{u'}{v'}_{L^2}-\Re \big( \g u(0) \overline v(0) - i\a|v(0)|^2 \big)\\
&=(\m^2-m^2) \innp v u _{L^2} ,
\end{align*}
and so
\begin{align*}
|\langle \Ac U, U \rangle_{\HH,\m,\g}|
=(\m^2-m^2)|\langle u,v\rangle_{L^2}|\leq \frac{\m^2-m^2}{2}(\nr{u}^2_{L^2}+\nr{v}^2_{L^2}).
\end{align*}
On the other hand, by Lemma \ref{lem-norme-H1}, 
\[
\nr{U}_{\HH,\m,\g}^2 \gtrsim \nr{u}_{L^2}^2 + \nr{v}_{L^2}^2.
\]
Hence, fixing $\b \geq 0$ large enough, we have
\begin{equation} \label{eq-Ac-diss}
\innp{(\pm \Ac-\b)U}{U}_{\HH,\m,\g} \leq 0.
\end{equation}
Therefore, by \cite[Proposition 2.4.2]{CH98}, the operators $\pm \Ac-\b$ are dissipative. 
In particular, for $\l > \b$, we have
\begin{equation} \label{minor-Ac-l}
\nr{(\pm\Ac-\l)U}_{\HH,\m,\g}^2 \geq \nr{(\pm\Ac-\b) U}_{\HH,\m,\g}^2 + (\l-\b)^2 \nr{U}_{\HH,\m,\g}^2,
\end{equation}
so that $\pm \Ac-\l$ are injective with closed range. 
Now, let $F = (f,g) \in \HH$. For $U = (u,v) \in \DD$, we have 
\[
(\Ac - \l) U = F \eqv \begin{cases} v = \l u + f, \\ u'' - (m^2 + \l^2) u = g + \l f. \end{cases} 
\]
By Lemma \ref{lem-res-u}, if $\b$ is large enough, there exists $U = (u,v) \in \DD$ such that 
the right-hand side is satisfied. It is given by $u = R(\l) (-g - \l f - i \a \d f)$ and $v = \l u + f$. 
This proves that $\text{Ran} (\Ac - \l) = \HH$. Hence, $(\Ac - \l)$ has a bounded inverse and, by \eqref{minor-Ac-l},
\[
\nr{(\Ac-\l)\inv}_{\Lc(\HH)} \leq \frac 1 {\l-\b}.
\]
By the Hille--Yosida Theorem, this proves that $\Ac$ generates a $C^0$-semigroup on 
$(\HH,\nr\cdot_{\HH,\m,\g})$. Furthermore, for $t \geq 0$ and $U \in \HH$, we have 
\[
\nr{e^{t\Ac}U}_{\HH,\m,\g} \leq e^{\b t} \nr{U}_{\HH,\m,\g}.
\]
Since the norm $\nr{\cdot}_{\HH,\m,\g}$ is equivalent to the usual one, there exists $M \geq 1$ 
such that we also have 
\[
\nr{e^{t\Ac}U}_{\HH} \leq Me^{\b t} \nr{U}_{\HH}.
\]
Now the same holds true with $\Ac$ replaced by $-\Ac$, and the proof is complete.
\end{proof}

\begin{remark}
\rm
If $\g>-m$, we can chose $\m=m$ in Lemma \ref{lem-norme-H1} and \ref{lem-res-u}. 
As in the proof of Proposition \ref{prop-semigroups}, we can show that $\Ac$ is skew-adjoint and, by \cite[Theorem 3.2.3]{CH98}, it now generates a one-parameter unitary group:
\[
\nr{e^{t\Ac} U}_{\HH,m,\g}=\nr{U}_{\HH,m,\g}.
\]
\end{remark}


\subsection{Local well-posedness of the nonlinear problem}
We are now in a position to prove the local well-posedness of 
the Cauchy problem \eqref{eq-wave-0}. 
We suppose that the general nonlinearity $f\in C(\C,\C)$ satisfies the following:
\begin{align}
f(0)&=0,
\\
|f(u)-f(v)|&\leq C_f(1+|u|^{p-1}+|v|^{p-1})|u-v|,
\end{align}
where $p\geq 1$ and $C_f>0$.

\begin{lemma} For any $R > 0$, there exists $C_R \geq 0$ such that, for $u_1,u_2 \in H^1$ 
with $\nr{u_1}_{H^1} \leq R$ and $\nr{u_2}_{H^1} \leq R$, we have 
\begin{align*}
\nr{f(u_1)}_{L^2} & \leq  C_R \nr{u_1}_{H^1},\\
\nr{f (u_1) - f (u_2)}_{L^2} & \leq  C_R \nr{u_1 - u_2}_{H^1}.
\end{align*}
\end{lemma}

\begin{proof}
For $j \in \{1,2\}$ we have $\nr{u_j}_{L^\infty} \leq R$, so with $C_f$ we get
\begin{align*}
\nr{f (u_1)-f (u_2)}_{L^2}^2
& \leq \int_\R C_f^2 \big(1 + \abs{u_1(x)}^{p-1} + \abs{u_2(x)}^{p-1} \big)^2 \abs{u_1(x) - u_2(x)}^2 \dif x\\
& \leq C_f^2 (1+2R^{p-1})^2 \nr{u_1-u_2}_{L^2}^2.
\end{align*}
This gives the second inequality. The first one follows by taking $u_2 = 0$.
\end{proof}

\begin{corollary} \label{Lipschitz}
$F:\HH\rightarrow\HH$ is Lipschitz continuous on bounded subsets of $\HH$: for any $R>0$, 
there is a constant $L(R)$ such that, for $U, V \in \HH$ with $\|U\|_\HH\leq R$ and $\|V\|_\HH\leq R$, 
we have
\[
\|F(U)-F(V)\|_\HH\leq L(R)\|U-V\|_\HH.
\] 
\end{corollary}

\begin{lemma} \label{uniqueness}
Let $T>0$, $U_0\in \HH$, and $U, V \in C([0,T],\HH)$ be two solutions to \eqref{eq-Duhamel}. Then $U=V$.
\end{lemma}

\begin{proof}
Let us set $R= \sup_{t\in[0,T]}\max\{ \nr{U(t)}_\HH, \nr{V(t)}_\HH\}$. For $t \in [0,T]$ we have
\[
\nr{U(t)-V(t)}_\HH \leq \int_0^t \nr{ e^{(t-s)\Ac}\Big(F(U(s))-F(V(s))\Big) }_\HH \diff s \leq Me^{T\b} L(R) \int_0^t \nr{ U(s)-V(s) }_\HH \diff s.
\]
By Gronwall's Lemma we conclude that $U(t)=V(t)$ for all $t \in [0,T]$.
\end{proof}

In the next proposition we prove the existence of a weak solution to the Cauchy problem. 

\begin{proposition} \label{existence}
Take $R>0$ and $U_0\in\HH$ such that $\|U_0\|_\HH\leq R$. 
Then there exists $T_R>0$ and a unique solution $U\in C([0,T_R),\HH)$ of problem \eqref{eq-Duhamel}. 
\end{proposition}

\begin{proof}
We only need to prove the existence of the solution, as uniqueness follows from Lemma \ref{uniqueness}. Let $M$ and $\b$ be as in Proposition \ref{prop-semigroups}.  
Consider $U_0 \in \HH$ such that $\|U_0\|_\HH\leq R$. For $T>0$ to be determined later, let
\[
X:=\{ U\in C([0,T],\HH) : \|U(t)\|_\HH\leq 3MR \ \forall t\in [0,T] \}
\]
and
\[
d(U,V):=\max_{t\in [0,T]} \| U(t)-V(t)\|_\HH. 
\]
Then $(X,d)$ is a complete metric space. 
We now define a map $\Psi:X\rightarrow C([0,T],\HH)$ by
\[
\Psi(U) : t \mapsto  e^{t\Ac}U_0+\int_0^t e^{(t-s)\Ac} F(U(s)) \diff s.
\]
Note that, for all $s\in[0,T]$, we have $\| F(U(s)) \|_\HH\leq 3MR \, L(3MR) $ 
by Corollary \ref{Lipschitz}. Thus,
\begin{align}
\|\Psi(U)(t) \|_\HH &\leq \| e^{t\Ac} U_0 \|_\HH + \int_0^t \| e^{(t-s)\Ac} F(U(s))\|_\HH \diff s \nonumber
\\
&\leq Me^{T \b}R+Me^{T \b} \int_0^t \| F(U(s)) \|_\HH  \diff s \nonumber
\\
&\leq Me^{T \b}R+Me^{T \b} 3MR \, L(3MR) \, T.
\end{align}
Furthermore, for $U, V \in E$, we have
\begin{align*}
\nr{ \Psi(U)(t)-\Psi(V)(t)}_\HH &\leq M e^{T \b} \int_0^t \nr{F(U(s))-F(V(s))}_\HH \diff s
\\
&\leq Me^{T\b} L(3MR)\, T\, d(U,V).
\end{align*}
It is now straightforward to check that, if $T=T_R>0$ is chosen small enough, there holds
\begin{align*}
Me^{T \b}R+&Me^{T \b} 3MR \, L(3MR)T \leq 3MR,
\\
M&e^{T \b} L(3MR) T \leq 1/3.
\end{align*}
This shows that $\Psi$ maps $(X,d)$ to itself and is a contraction.
The result now follows from the Fixed Point Theorem.
\end{proof}

\begin{theorem} \label{blow-up}
There exists a function $T:\HH\rightarrow (0,\infty]$ with the following properties. For all $U_0\in \HH$, 
there exists a function $U\in C([0,T(U_0)),\HH)$ such that, for all $0<T<T(U_0)$, 
$U$ is the unique solution of \eqref{eq-Duhamel} in $C([0,T],\HH)$. 
Furthermore, the blow-up alternative holds: if $T(U_0)<\infty$ then 
$\lim_{t\uparrow T(U_0)}\nr{U(t)}_\HH=\infty$. 
\end{theorem}

\begin{proof}
For all $U_0\in\HH$, we set
\[
T(U_0)=\sup \{ T>0 : \exists U\in C([0,T],\HH) \text{ solution to } \eqref{eq-Duhamel}\}.
\]
From Proposition~\ref{existence}, we know that $T(U_0)> T_{\nr{U_0}} >0$ and Lemma \ref{uniqueness} allows us to extend it to a maximal solution $U\in C([0,T(U_0)],\HH)$. 
The blow-up alternative follows from an argument by contradiction. 
Suppose that $T(U_0)<\infty$ and that there exists a constant $C$ and a sequence 
$t_n$ in $[0,T(U_0))$ such that $t_n\uparrow T(U_0)$ and $\sup_{n\in\N}\nr{U(t_n)}_\HH\leq C$. 
Now take a time $t_n$ such that $t_n+T_C>T(U_0)$. Using Lemma~\ref{uniqueness} and 
Proposition~\ref{existence}, we can extend the solution up to $t_n+T_C$ by considering 
the initial value problem \eqref{eq-Duhamel} with initial value $U(t_n)$. 
This contradicts the definition of $T(U_0)$ and concludes the proof.
\end{proof} 

\begin{theorem} \label{semicontinuity}
Following the notation of Theorem \ref{blow-up}, we have the following properties:
\begin{enumerate}[\rm (i)]
\item 
$T:\HH\rightarrow (0,\infty]$ is lower semicontinuous: given the initial conditions $U_0, U_{0,n} \in\HH$ such that $U_{0,n}$ converges to $U_0$ in $\HH$, we have that
\[
T(U_0)\leq \liminf_{n\rightarrow\infty} T(U_{0,n}).
\]
\item 
If $U_{0,n}\rightarrow U_0$ and if $T<T(U_0)$, then $U_n\rightarrow U$ in $C([0,T],\HH)$, where $U_n$ and $U$ are the solutions of \eqref{eq-Duhamel} corresponding to the initial data $U_{0,n}$ and $U_0$.
\end{enumerate}
\end{theorem}

\begin{proof}
Let $U_0\in\HH$ and $U\in C([0,T(U_0)),\HH)$ be the solution of \eqref{eq-Duhamel} 
given by Theorem~\ref{blow-up}. Let $0<T<T(U_0)$. 
It suffices to show that, if $U_{0,n}\rightarrow U_0$ then $T(U_{0,n})>T$ for $n$ large enough, 
and $U_n\rightarrow U$ in $C([0,T],\HH)$. We set
\[
R=2\sup_{t\in [0,T]} \nr{U(t)}_\HH,
\]
and
\[
\t_n=\sup\{t\in [0, T(U_{0,n})) : \nr{U_n(s)}_\HH \leq 2R \ \forall s \in [0,t]\}.
\]
If $n$ is large enough, we have $\nr{U_{0,n}}_\HH\leq R$. Hence, by Proposition~\ref{existence}, 
$0<T_R<\t_n$. Now, for all $0<t\leq \min\{ \t_n, T \}$,
\begin{align*}
\nr{U(t)-U_n(t)}_\HH &\leq \nr{e^{t\Ac} (U_0-U_{0,n}) }_\HH + \int_0^{t} \nr{e^{(t-s)\Ac}(F(U(s))-F(U_n(s)))}_\HH \diff s
\\
&\leq Me^{T\b}\nr{U_0-U_{0,n}}_\HH + Me^{T\b} L(2R) \int_0^{T} \nr{ U(s)-U_n(s) }_\HH \diff s.
\end{align*}
Therefore, by Gronwall's lemma,
\begin{equation} \label{gronwall}
\nr{U(t)-U_n(t)}_\HH \leq Me^{T\b} \nr{U_0-U_{0,n}}_\HH e^{{Me^{T\b}}L(2R) T}
\end{equation}
for all $t\leq \min \{T,\t_n\}$. In particular, if $n$ is large enough,
\[
\nr{U_n(t)}_\HH \leq R
\]
for $t\leq \min\{T,\t_n\}$. Hence $\t_n\geq T$, which implies that $T(U_{0,n})>T$. 
From \eqref{gronwall} we also see that $U_n\rightarrow U$ in $C([0,T],\HH)$, which completes the proof. 
\end{proof}

\begin{theorem}
Let $U_0\in\DD$ and $T \in (0,T(U_0))$. 
Let $U\in C([0,T],\HH)$ be the corresponding solution of \eqref{eq-Duhamel}. Then 
$U\in C([0,T],\mathfrak{D})\cap C^1([0,T],\HH)$.
\end{theorem}

\begin{proof} Let $h>0$ and $t\in [0,T-h)$. By a change of variables it is easy to see that
\[
U(t+h)-U(t)=e^{(t+h)\Ac}U_0-e^{t\Ac}U_0 + \int_0^t e^{s\Ac} \big( F(U(t+h-s))-F(U(t-s)) \big) \diff s + \int_0^h e^{(t+h-s)\Ac} F(U(s)) \diff s.
\]
Hence,
\begin{align*}
\nr{U(t+h)-U(t)}_\HH &\leq \nr{e^{t\Ac} (e^{h\Ac}U_0-U_0)}_\HH + \int_0^t\nr{e^{s\Ac}{ \big(F(U(t+h-s))-F(U(t-s))\big) }}_\HH \diff s
\\ 
&+ \int_0^h\nr{e^{(t+h-s)\Ac} F(U(s)) }_\HH \diff s
\\
& \leq Me^{T\b}\nr{e^{h\Ac}U_0-U_0}_\HH + Me^{T\b}L(R)\int_0^t \nr{ U(t+h-s)-U(t-s)}_\HH \diff s
\\
& + Me^{T\b} h\sup_{s\in[0,T]} \nr{F(U(s))}_\HH.
\end{align*}
We know that
\[
e^{h\Ac}U_0-U_0=\int_0^h e^{s\Ac} \Ac U_0 \diff s,
\]
and so $\nr{ e^{h\Ac}U_0-U_0 }_\HH \leq hMe^{\b T} \nr{ \Ac U_0 }_\HH$. Applying Gronwall's Lemma, 
we get
\[
\nr{U(t+h)-U(t) }_\HH \lesssim h
\]
for all $0\leq t<t+h\leq T$. Hence, $U: [0, T]\rightarrow \HH$ and $ F(U): [0,T]\rightarrow \HH $ are Lipschitz continuous. We conclude by \cite[Corollary 1.4.41]{CH98} and \cite[Proposition 4.1.6]{CH98}.
\end{proof}

\begin{remark}
\rm 
It is worth noting that solutions of \eqref{eq-wave-0} may blow up in finite time. To this aim, 
let us consider the ordinary differential equation
\begin{equation}\label{ode}
v''(t)+v(t)-v^3(t)=0.
\end{equation}
For any fixed $T>0$, this equation has the solution 
\begin{equation}\label{ode_blowup}
v(t)=\frac{1}{\tanh\left(T-t/\sqrt{2}\right)},
\end{equation}
which blows up at time $\sqrt{2}T$.
Now, consider \eqref{eq-wave-0} with $m=1$, $\gamma=\alpha=0$ and $p=3$, and choose
the constant initial data $u_0=1/\tanh\left(T\right)$. By finite speed of propagation (see e.g.~\cite{CH98}), 
if $u_0$ is smoothly truncated outside an interval of length $2\sqrt{2}T+1$, the corresponding solution of 
\eqref{eq-wave-0} will blow up like \eqref{ode_blowup} at time $\sqrt{2}T$.
Again by finite speed of propagation, if the support of the truncated $u_0$ is chosen far away from $x=0$, 
then the solution $u$ will not `see' the Dirac potentials over the time interval $[0,\sqrt{2}T)$, 
and will also blow up at time $\sqrt{2}T$, for any values of $\gamma$ and $\alpha$. 
\end{remark}

 
\section{Hamiltonian structure}\label{ham.sec}

In this section we show that \eqref{eq-wave-0} is a Hamiltonian system, and
we establish the relevant conservation laws, namely that the energy and the charge defined in
\eqref{energy} and \eqref{charge} are constants of the motion.
We shall use the general framework developed in \cite{DeGeRo15} to study orbital stability of 
standing waves of infinite-dimensional Hamiltonian systems. 

We start by showing that,
in the terminology of \cite[Sec.~6]{DeGeRo15}, $(\HH,\DD,\Jc)$ forms an appropriate 
symplectic Banach triple for our problem,
provided the map $\Jc:\HH\to\HH^*$ defined by
\begin{equation}\label{symplector}
\Jc(u,v)=(-i\a\d u-v,u)
\end{equation}
is a (weak) symplector, in the sense of Definition~6.2~(i)
in \cite{DeGeRo15}, which we check now.

\begin{lemma}\label{symplector.lem}
The map $\Jc:\HH\to\HH^*$ defined by \eqref{symplector} is a symplector, that is:
\begin{enumerate}[\rm (i)]
\item $\Jc$ is a bounded linear map;
\item $\Jc$ is one-to-one;
\item $\Jc$ is anti-symmetric, in the sense that
\[
\ps{\Jc(u,v),(w,z)}=-\ps{\Jc(w,z),(u,v)}, \quad (u,v),(w,z) \in \HH.
\]
\end{enumerate}
\end{lemma} 

\begin{proof}
(i) Linearity is obvious and boundedness follows from the Sobolev embedding theorem through the estimate
\begin{align*}
|\ps{\Jc(u,v),(w,z)}|
&=\Big|\Im\a u(0)\bar w(0)-\Re\intr v \bar w+\Re\intr u\bar z\Big|\\
&\leq |\a| \norm{u}_{H^1}\norm{w}_{H^1}+\norm{v}_{L^2}\norm{w}_{L^2}+\norm{u}_{L^2}\norm{z}_{L^2}\\
&\lesssim \big(\norm{u}_{H^1}+\norm{v}_{L^2}\big)\big(\norm{w}_{H^1}+\norm{z}_{L^2}\big).
\end{align*}

(ii) $\Jc$ is one-to-one since, clearly, $\Jc(u,v)=(0,0)$ if and only if $(u,v)=(0,0)$.

(iii) The antisymmetry of $\Jc$ follows by a straightforward calculation, using that $\a\in\real$.
\end{proof}

We now turn our attention to the regularity of the energy and 
charge functionals respectively introduced in
\eqref{energy} and \eqref{charge}. In particular, in the terminology of Definition~6.5 in \cite{DeGeRo15}, 
we show that $E$ and $Q$ have $\Jc$-compatible derivatives, i.e.~that 
$E'(u,v),Q'(u,v) \in \rge\Jc$ for all $(u,v)\in\DD$. We write $E,Q\in \mathrm{Dif}(\Dc,\Jc)$.

\begin{lemma}\label{energy_charge.lem}
We have that $E\in C^1(\HH,\real)\cap C^2(\DD,\real)$ and $Q\in C^2(\HH,\real)$. 
For $(\f,\p)\in\DD$ and $(u,v),(w,z)\in\HH$, we have 
\[
E'(\f,\p)=(-\f''+m^2 \f-|\f|^{p-1}\f-i\a\d \p, \p), 
\]
\begin{multline*}
\ps{E''(\f,\p)(u,v),(w,z)}
= \Re\Big[\g u(0)\bar w(0)+\intr \Big( u'\bar w'+\big(m^2 u-|\f|^{p-1}u-(p-1)|\f|^{p-3}\f\re(\f\bar{u})\big)\bar w \Big) \diff x
+\intr v\bar z \diff x\Big], 
\end{multline*}
and for $(\f,\p),(u,v)\in\HH$
\[
Q'(\f,\p)=(-\a\d \f+i\p,-i\f),
\]
\[
Q''(\f,\p)(u,v)=(-\a\d u+iv,-iu).
\]
Furthermore, $E'(\f,\p)\in\rge\Jc$ and $Q'(\f,\p)\in\rge\Jc$ for all $(\f,\p)\in\HH$.
\end{lemma}

\begin{proof}
The regularity stated and the expressions obtained for the Fr\'echet derivatives follow from routine verifications.

Let $(\f,\p)\in\DD$. To see that $E'(\f,\p)\in\rge\Jc$, one has to find $(w,z)\in\HH$ such that
\[
-i\a\d w-z=-\f''+m^2 \f-|\f|^{p-1}\f-i\a\d \p \qandq w=\p \quad \text{in} \ \HH^*.
\]
This yields $(w,z)=(\p,\f''-m^2 \f+|\f|^{p-1}\f)$, which clearly belongs to $\HH$.
Similarly, for $(\f,\p) \in \HH$,
\begin{equation}\label{JinverseQ}
\Jc(w,z)=Q'(\f,\p) \iff (w,z)=-i(\f,\p).
\end{equation}
This completes the proof.
\end{proof}

Lemmas \ref{symplector.lem} and \ref{energy_charge.lem} show that $(\HH,\DD,\Jc)$ is a suitable
symplectic Banach triple for our problem, with associated Hamiltonian $E$.
For initial conditions $U_0=(u_0,u_1)\in \DD$, the differential equation in \eqref{eq-wave-0} can indeed 
be written as the Hamiltonian system (see Definition~6.6 in \cite{DeGeRo15})
\begin{equation}\label{ham-0}
\Jc\frac{\dif}{\dif t}U(t)=E'(U(t)).
\end{equation}

\begin{remark}\label{stable_domain.rem}
\rm The well-posedness theory in Section~\ref{cauchy.sec} shows that the domain $\DD$ is stable
under the flow of \eqref{ham-0}, so that, by Lemma~\ref{energy_charge.lem}, 
$E'(U(t))$ indeed belongs to $\rge\Jc$ over the lifespan of the solution.
\end{remark}

\begin{proposition}
The energy $E$ and the charge $Q$ are constants of the motion for \eqref{ham-0}, i.e.~for any 
$U_0=(u_0,u_1)\in \HH$, $E(U(t))=E(U_0)$ and $Q(U(t))=Q(U_0)$, as long as the solution exists.
\end{proposition}

\begin{proof}
Following \cite[Theorem~5,~p.~191]{DeGeRo15}, one only needs to check that both $E$ and $Q$
Poisson-commute with $E$, i.e.~that $\{E,E\}(u,v)=\{E,Q\}(u,v)=0$ for all $(u,v)\in\DD$, where for any
$F\in \mathrm{Dif}(\DD,\Jc)$, the Poisson bracket $\{E,F\}$ is defined as
\begin{equation}\label{poisson}
\{E,F\}(u,v)=\ps{E'(u,v),\Jc^{-1}F'(u,v)}, \quad (u,v)\in\DD.
\end{equation}
That $\{E,E\}(u,v)=0$ for all $(u,v)\in\DD$ is a trivial consequence of the anti-symmetry of $\Jc$.
As for $\{E,Q\}$, using the explicit expression $\Jc^{-1}(w,z)=(z,-w-i\a\d z)$ (or \eqref{JinverseQ}), 
we have
\begin{align*}
\{E,Q\}(u,v)
&=\ps{(-u''+m^2 u-|u|^{p-1}u-i\a\d v, v),(-iu,-iv)} \\
&=\Re\Big[\intr (-u'' + m^2 u - |u|^{p-1}u -i\a\d v)\overline{(-iu)} \diff x + \Re\intr v\overline{(-iv)}\diff x\Big]\\
&=\Re\Big[-i\intr u''\bar u \diff x + \al v(0)\bar u(0)\Big]\\
&=\Re[-i\bar u(0)(u(0^-)-u(0^+))+\al \bar u(0)v(0)]=0,
\end{align*}
which completes the proof.
\end{proof}


\section{Stability of standing waves} \label{spectrum.sec}

Having established the well-posedness and the Hamiltonian structure of 
the initial-value problem \eqref{eq-wave-0}, 
we now investigate the stability of standing waves by applying the energy-momentum method
described in the introduction. The criterion for orbital stability of the standing waves \eqref{stwave}
is the following.

\begin{proposition} \label{gss1}
Suppose the standing wave $e^{i\omega_0 t}\Phi_{\omega_0}(x)$
satisfies the spectral conditions \textnormal{(S1)--(S3)}. 
Then it is orbitally stable if 
\[
\frac{\diff}{\diff\omega}\Big\vert_{\om=\om_0}Q(\Phi_\omega)>0,
\]
and orbitally unstable if
\[
\frac{\diff}{\diff\omega}\Big\vert_{\om=\om_0}Q(\Phi_\omega)<0.
\]
\end{proposition}

Let $A$ be a selfadjoint operator that is bounded below with positive essential spectrum.
We shall henceforth denote by $n(A)\in\N$ the number of negative eigenvalues 
(counted with multiplicities) of $A$,
and we set
\[
n_{\omega}:=n \big( \tilde{R}^{-1}L''_{\omega}(\Phi_{\omega}) \big),
\]
for all admissible $\omega\in\R$. 
In Proposition~\ref{gss1}, we have $n_{\omega_0}=1$. 
For $n_{\omega_0}=2$, we will exhibit regimes of linear instability  
using the following criterion, borrowed from \cite{Gs2}.

\begin{proposition} \label{gss2}
Let
\[
p(d''(\o))=\begin{cases}
1 \quad \textit{if} \ \frac{\diff }{\diff \o}\Big\vert_{\o=\o_0} Q(\Phi_\o) > 0, \\
0 \quad \textit{if} \ \frac{\diff }{\diff \o}\Big\vert_{\o=\o_0} Q(\Phi_\o) < 0.
\end{cases} 
\]
Then the standing wave $e^{i\omega_0 t}\Phi_{\omega_0}$ is linearly unstable if $n_{\o_0}-p(d''(\o_0))$ is odd.
\end{proposition}

\begin{corollary} Suppose the standing wave $e^{i\omega_0 t}\Phi_{\omega_0}$ 
satisfies \textnormal{(S$1'$)--(S3)} and
\[
\frac{\diff }{\diff \o}\Big\vert_{\o=\o_0} Q(\Phi_\o) > 0.
\]
Then it is linearly unstable.
\end{corollary}


\subsection{Spectral analysis}  \label{PointSpectrum}

Our purpose here is to give some spectral properties 
(in particular the number of negative eigenvalues) of 
the operator $\tilde R \inv L''_\omega(\Phi_\omega)$. We will consider
$\alpha, \gamma$ and $\omega$ satisfying \eqref{omega_range}. 
The quantity 
\[
\beta=\beta(\omega):=\g - \a \o, 
\]
which appears in \eqref{stat}, will play an important role below. 
In view of the admissibility condition \eqref{omega_range}, we shall consider
$\b \in (-\b_0,\b_0)$, where $\b_0 := 2 \sqrt{m^2 - \o^2}$.
The main results of this subsection rely on
the dependence on $\beta$ of the key objects entering the spectral analysis.
With this in mind (and to avoid a too heavy notation)
we shall relabel various quantities by $\beta$ and temporarily drop the index $\omega$.
For instance --- with a slight abuse of notation --- we will write $\f_\b$ instead of $\f_\o$, that is,
\begin{equation} \label{def-tilde-phi}
\f_\b(x) = 
\left[\frac{(p+1)(m^2-\omega^2)}{2}\textnormal{sech}^2
\left(\frac{(p-1)\sqrt{m^2-\omega^2}}{2}|x|+\tanh^{-1}
\left(-\frac{\b}{2\sqrt{m^2-\omega^2}}\right)\right)\right]^{\frac{1}{p-1}}.
\end{equation}
One should of course keep in mind the dependence on $\omega$. It will not be relevant
for our analysis here, but will come back with full force in the next subsection.
For $U = (u,v) \in \HH$, we now let 
\begin{equation} \label{expr-Lb}
\Lpp U := \big( -u'' + m^2 u -  \f_\b^{p-1} u - (p-1) \f_\b^{p-1} \Re(u) + \b \d u + i \o v , v -i\o u \big) \quad \in \HH^*.
\end{equation}
In view of Lemma~\ref{energy_charge.lem}, this reads $\Lpp=L_\o''(\Phi_\o)$. 
We shall also use the convenient notation $\RLpp := \tilde R\inv \Lpp$.

Let $\b \in (-\b_0,\b_0)$. We observe that $\Lpp : \HH \to \HH^*$ is a bounded operator and, 
for $U, W \in \HH$, we have 
\begin{equation} \label{eq-L-sym}
\langle \RLpp U , W \rangle_\HH = \langle \Lpp U,W \rangle_{\HH^*,\HH} 
= \langle U,\Lpp W \rangle_{\HH,\HH^*} = \langle U,\RLpp W \rangle_\HH,
\end{equation}
so $\RLpp$ is a bounded self-adjoint operator on $\HH$.

Instead of analyzing directly the spectral properties of $\RLpp$, it will be more convenient to work with the operator on $L^2 \times L^2$ associated to the form $\Lpp$. More precisely, we set
\begin{equation} \label{def-Db}
D_\b :=\{ u \in \Hds : u'(0^+)-u'(0^-)= \b u (0) \}
\end{equation}
and we consider on $L^2 \times L^2$ the operator $\Lcb$ defined by $\Dom(\Lcb) = D_\b \times L^2$ and, for $U = (u,v) \in \Dom(\Lcb)$,
\[
\Lcb U = \big( -u'' + m^2 u -  \f_\b^{p-1} u - (p-1)  \f_\b^{p-1} \Re(u) + i \o v , v -i\o u \big) \quad \in L^2 \times L^2.
\]
This defines a ($\R$-linear) self-adjoint operator which shares the same relevant spectral properties as $\RLpp$:

\begin{proposition} \label{prop-RLpp-Lcb}
The operator $\Lcb$ is selfadjoint and bounded from below on $L^2 \times L^2$, and for $U \in D_\b \times L^2$ we have 
\[
\innp{\Lcb U}{U}_{L^2 \times L^2} = \langle\Lpp U,U\rangle_{\HH^*,\HH} = \langle \RLpp U,U\rangle_{\HH}. 
\]
We have 
\begin{equation} \label{eq-ker}
\ker(\Lcb) = \ker(\RLpp).
\end{equation}
Moreover
\begin{equation} \label{eq-Sp-ess-pos}
\inf \s_\ess (\Lcb) > 0 \quad \Longleftrightarrow \quad \inf \s_\ess (\RLpp) > 0,
\end{equation}
and, in this case, 
\begin{equation} \label{eq-eg-n}
n(\Lcb) = n(\RLpp).
\end{equation}
\end{proposition}

\begin{proof}
\stepp For $U,V \in \HH$, we set $Q(U,V) = \langle\Lpp U,V\rangle_{\HH^*,\HH}$. This defines on $L^2 \times L^2$ a symmetric bilinear form with domain $\Dom(Q) = \HH$. Using a trace inequality as in \eqref{trace}, we can check that $Q$ is bounded from below and closed. We denote by $T$ the corresponding selfadjoint operator given by the Representation Theorem (see for instance \cite[VI-Theorem~2.1]{Kato} for sesquilinear complex forms, the symmetric case being analogous for real bilinear forms). In particular,
\[
\Dom(T) = \left\{ U \in \HH \st V \mapsto Q(U,V) \text{ is a continuous linear functional on } L^2 \times L^2 \right\},
\]
and 
\begin{equation} \label{eq-TUV}
\innp{TU}{V}_{L^2 \times L^2} = Q(U,V), \quad \text{for all} \ U \in \Dom(T), V \in \HH.
\end{equation}
Since $\innp{\Lcb U}{V}_{L^2 \times L^2} = Q(U,V)$ for all $U \in \Dom(T)$ and $V \in \HH$, we have $\Dom(\Lob) \subset \Dom(T)$ and $\Lob = T$ on $\Dom(\Lob)$. Now let $U = (u,v) \in \Dom(T)$. Writing \eqref{eq-TUV} with $V = (w,0)$ for any $w \in C_0^\infty(\R^*)$ proves that $u \in \Hds$. Then, with $w \in C_0^\infty(\R)$ such that $w(0) = 1$ or $w(0) = i$, we obtain $u \in D_\b$, so $U \in \Dom(\Lcb)$. This means that $\Lcb = T$, and the first part of the proposition is proved.

\stepp Let $U = (u,v) \in \ker(\Lcb)$. In particular we have $U \in \HH$ and 
$\langle \RLpp U,W \rangle_{\HH} = \langle \Lpp U,W \rangle_{\HH^*,\HH} = 0$ for all $W \in \HH$, so $U \in \ker(\RLpp)$. 
Conversely, if $U \in \ker(\RLpp )$
then $\langle\Lpp  U,W \rangle = \innp{0}{W}_{L^2 \times L^2}$ for all $W \in \HH$, 
so $U \in \Dom(\Lcb)$ and $\Lcb U = 0$. This proves \eqref{eq-ker}.

\stepp Now suppose that $\inf \s_\ess(\Lcb) > 0$. 
Since $\Lcb$ is bounded from below, it has a finite number $\tilde m$ 
of non-positive eigenvalues (counted with multiplicities). 
We denote by $\tilde \Th$ the subspace of $L^2 \times L^2$ generated by the 
corresponding eigenvectors, and by $\tilde \Th^\bot$ the orthogonal complement 
of $\tilde \Th$ in $L^2 \times L^2$. We also set $\tilde \Th_1^\bot = \tilde \Th^\bot \cap \HH$. 
Since $\tilde \Th \subset \Dom(\Lcb) \subset \HH$, $\tilde \Th$ and $\tilde \Th_1^\bot$ 
are complementary subspaces of $\HH$.

There exists $\s_0 > 0$ such that $\s(\Lcb) \cap (0,\s_0) = \emptyset$. Then, 
for all $U \in \Dom(\Lcb) \cap \tilde \Th^\bot$, we have 
\[
\langle \RLpp U , U \rangle_{\HH} = \langle \Lpp  U , U \rangle_{\HH^*,\HH} 
= \langle \Lcb U , U \rangle_{L^2 \times L^2} \geq \s_0 \nr{U}_{L^2 \times L^2}^2.
\]
On the other hand, by the trace inequality there exists $C \geq 0$ such that, for all $U \in \HH$,
\[
\langle \Lpp  U,U\rangle_{\HH^*,\HH} \geq \frac {\nr{U}_\HH^2} 2 - C \nr{U}_{L^2 \times L^2}.
\]
Thus, for $\eta \in (0,1)$ and $U \in \Dom(\Lcb) \cap \tilde \Th_1^\bot$, we have 
\[
\langle \RLpp U,U\rangle_{\HH} \geq \frac {\eta}2 \nr{U}_{\HH}^2 - \eta C \nr{U}_{L^2 \times L^2} 
+ (1-\eta) \s_0 \nr{U}_{L^2 \times L^2}^2.
\]
For $\eta > 0$ small enough, this yields 
\[
\langle\RLpp U,U\rangle_{\HH} 
\geq \frac {\eta}2 \nr{U}_{\HH}^2, \quad\text{for all} \ U \in \Dom(\Lcb) \cap \tilde \Th_1^\bot.
\]
But $\Dom(\Lcb) \cap \tilde \Th_1^\bot$ is dense in $\Th_1^\bot$, so
\[
\langle\RLpp U,U\rangle_{\HH} \geq \frac {\eta}2 \nr{U}_{\HH}^2,
\quad\text{for all} \ U \in \tilde \Th_1^\bot.
\]
Since $\tilde \Th_1^\bot$ is of codimension $m$ in $\HH$, 
the Min-Max Principle (see, e.g., Theorem XIII.1 in \cite{rs4}) implies
that $\inf \s_\ess(\RLpp) > 0$ and that $\RLpp$ 
has at most $\tilde m$ negative eigenvalues (counted with multiplicities).

Conversely, assume that $\inf \s_\ess(\RLpp) > 0$. Since $\RLpp$ is bounded, 
it has a finite number $m$ of non-positive eigenvalues (counted with multiplicities). 
We denote by $\Th$ the subspace of $\HH$ generated by the corresponding eigenvectors, 
and by $\Th^\bot$ the orthogonal complement of $\Th$ in $\HH$. 
There exists $\s_1 > 0$ such that $\s(\RLpp) \cap (0,\s_1) = \emptyset$. 
Then, for all $U \in \Th^\bot$, we have 
\[
\langle \Lpp  U, U \rangle_{\HH^*,\HH} 
= \langle \RLpp U, U \rangle_{\HH} \geq \s_1 \nr{U}_\HH^2 \geq \s_1 \nr{U}_{L^2 \times L^2}^2.
\]
We recall that $\HH$ is the form domain of $\Lcb$ and that $\Lcb$ 
is associated to the form $Q$, 
so by the form version of the Min-Max Principle (see Theorem XIII.2 in \cite{rs4}), 
we have $\inf \s_\ess (\Lcb) \geq \s_1 > 0$ and $\Lcb$ 
has at most $m$ non-positive eigenvalues (counted with multiplicities). 

We have thus proved \eqref{eq-Sp-ess-pos} and that, in this case, the operators $\RLpp$ and $\Lcb$ have 
the same number of non-positive eigenvalues. With \eqref{eq-ker}, this gives \eqref{eq-eg-n}.
\end{proof}

Since $\Lcb$ is not $\C$-linear, it is usual to split functions into real and imaginary parts. Then, the operator $\Lcb$ acting on pairs of complex-valued functions is formally equivalent to the following operator acting on quadruplets of real-valued functions:
\[
\begin{pmatrix}
\Lp + \o^2 & 0 & 0 & -\omega \\
0  & \Lm + \o^2 & \omega & 0 \\
0 &\omega& 1 & 0 \\
-\omega& 0 & 0& 1\\
\end{pmatrix},
\]
where 
\begin{align}
\Lp u &=-u''+(m^2-\omega^2)u-p  \f^{p-1}_{\b} u \label{op1},
\\
\Lm u &= -u'' + (m^2-\omega^2)u -  \f^{p-1}_{\b} u \label{op2}.
\end{align}
Here, $\Lp$ and $\Lm$ are $\R$-linear operators acting on a space of real-valued functions. However, we are going to use some spectral arguments which are more conveniently written with complex operators.

We denote by $L^2_\C$ the Lebesgue space $L^2(\R,\C)$ endowed with its usual complex structure. Then we define $H^1_\C$ and $H^{2,*}_\C$ accordingly. We also define $D^\b_\C$ as $D_\b$, with $\Hds$ replaced by $H^{2,*}_\C$. Then we define the operators $\Lp$ and $\Lm$ by $\Dom(\Lp) = \Dom(\Lm) = D^\b_\C \times L^2_\C$ and, for $u$ in $D^\b_\C \times L^2_\C$, $\Lp u$ and $\Lm u$ are defined by \eqref{op1} and \eqref{op2}. These are in particular $\C$-linear operators.\\

For $\l \in \R \setminus \{1\}$ we set (see Figure \ref{fig:Lambda})
\[
\Lambda(\lambda):=\lambda+\frac{\lambda\omega^2}{1-\lambda}.
\]

\begin{proposition} \label{prop-Lcb-Lp-Lm}
The operators $\Lp$ and $\Lm$ are selfadjoint and bounded from below on $L^2_\C$. Moreover, for $\l \in \R \setminus \{1\}$,
\begin{enumerate}[\rm (i)]
\item $\l \in \s(\Lob)$ if and only if $\L(\l) \in \s(\Lp) \cup \s(\Lm)$,
\item we have 
\begin{equation} \label{eq-dim-ker}
\dim(\ker( \Lob-\l)) = \dim(\ker( \Lp-\L(\l))) + \dim(\ker( \Lm-\L(\l))),
\end{equation}
and in particular,
\begin{equation} \label{eq-nL}
n(\RLpp) = n(\Lp) + n(\Lm),
\end{equation}
\item $\l \in \s_\ess(\Lob)$ if and only if $\L(\l) \in \s_\ess(\Lp) \cup \s_\ess(\Lm)$. 
\end{enumerate}
\end{proposition}

In \eqref{eq-dim-ker} and \eqref{eq-nL}, the left-hand sides are dimensions of real vector spaces while the right-hand sides are dimensions of complex vector spaces.

\begin{proof}
\stepp As in the proof of Proposition \ref{prop-RLpp-Lcb}, we can check that $\Lp$ and $\Lm$ are the selfadjoint operators corresponding to the sesquilinear forms 
\begin{equation} \label{def-qp}
q_\b^+ : (u,w) \mapsto \innp{u'}{w'}_{L^2_\C} + (m^2-\o^2) \innp{u}{w}_{L^2_\C} 
- p \langle \f_\b^{p-1} u,w\rangle_{L^2_\C} + \b u(0) \bar w(0)
\end{equation}
and 
\begin{equation} \label{def-qm}
q_\b^- : (u,w) \mapsto \innp{u'}{w'}_{L^2_\C} + (m^2-\o^2) \innp{u}{w}_{L^2_\C} 
-  \langle \f_\b^{p-1} u,w\rangle_{L^2_\C} + \b u(0) \bar w(0),
\end{equation}
which are closed, symmetric and bounded from below.

\stepp Let $\l \in \R$. Let $U = (u,v)$ and $F = (f,g)$ in $L^2 \times L^2$. We write $u = u_1 + i u_2$ where $u_1$ and $u_2$ are real-valued. We use similar notation for $v$, $f$ and $g$. Then $u$ belongs to $D_\b$ if and only if $u_1$ and $u_2$ belong to $D^\b_\C$ and in this case
\begin{equation*} 
(\Lcb - \l ) U = F \quad \Longleftrightarrow \quad 
\left\{
\begin{array}{rcl}
(\Lp + \o^2 - \l) u_1 - \o v_2 &=&  f_1,\\
(\Lm + \o^2 - \l) u_2 + \o v_1 &=& f_2,\\
(1-\l) v_1 + \o u_2 &=& g_1,\\
(1-\l) v_2 - \o u_1 &=& g_2.
\end{array}
\right.
\end{equation*}
If $\l \neq 1$ this gives
\begin{equation} \label{systeme}
(\Lcb - \l ) U = F \quad \Longleftrightarrow \quad 
\left\{
\begin{array}{rcl}
(\Lp - \L(\l)) u_1 &=&  f_1 + \frac {\o g_2}{1-\l},\\
(\Lm - \L(\l)) u_2 &=& f_2 - \frac {\o g_1}{1-\l},\\
v_1 &=& \frac {g_1 - \o u_2}{1-\l},\\
v_2 &=& \frac {g_2 + \o u_1}{1-\l}.
\end{array}
\right.
\end{equation}

\stepp We set $K^+(\l) = \ker(\Lp - \L(\l))$ and denote by $K^+_\R(\l)$ the $\R$-linear subspace of real-valued functions in $K^+(\l)$. Given $u \in D^\b_\C$, we have $u \in K^+(\l)$ if and only if $\Re(u)$ and $\Im(u)$ are in $K^+_\R(\l)$. A family of linearly independent vectors in $K^+_\R(\l)$ is also a family of linearly independent vectors in $K^+(\l)$, so $\dim_\R (K^+_\R(\l)) \leq \dim_\C(K^+(\l))$. In particular, if the left-hand side is infinite, then so is the right-hand side. Now assume that $\dim_\R (K^+_\R(\l))$ is finite (possibly 0) and consider a basis $e = (e_1,\dots,e_m)$ of $K^+_\R(\R)$ (with $m \in \N$). Let $u \in K^+(\l)$. Then $\Re(u)$ and $\Im(u)$ belong to $K^+_\R(\l)$ and are $\R$-linear combinations of vectors in $e$, so $u$ is a $\C$-linear combination of vectors in $e$. This proves that 
\[
\dim_\C K^+(\l) = \dim_\R K^+_\R(\l).
\]
We similarly define $K^-(\l)$ and $K^-_\R(\l)$ and see that $\dim_\C K^-(\l) = \dim_\R K^-_\R(\l)$.

If $e_1^+,\dots,e_{m_+}^+$ are linearly independent vectors in $K^+_\R(\l)$ and $e_1^-,\dots,e_{m_-}^-$ are linearly independent vectors in 
$K^-_\R(\l)$ ($m_\pm$ may be zero), then 
\begin{equation} \label{basis-ker}
\left( e_1^+  , \frac {i \o e_{1}^+}{1-\l} \right) , \dots, 
\left( e_1^+  , \frac {i \o e_{m_+}^+}{1-\l} \right), 
\left(i e_1^- , - \frac {\o e_{1}^-}{1-\l} \right), \dots, 
\left(i e_1^- , - \frac {\o e_{m_-}^-}{1-\l} \right)
\end{equation}
is a family of linearly independent vectors in $\ker(\Lob-\l)$, so 
\begin{equation} \label{dim-ker}
\dim(\ker(\Lob)-\l) \geq \dim(K^+_\R(\l)) + \dim(K^-_\R(\l)).
\end{equation}
In particular, if the right-hand side is infinite, then so is the left-hand side. 
Now assume that the right-hand side is finite. If the above families span $K^+_\R(\l)$ 
and $K^-_\R(\l)$, then \eqref{basis-ker} span $\ker(\Lcb-\L(\l))$, 
so the inequality in \eqref{dim-ker} is an equality and \eqref{eq-dim-ker} is proved. 
Since $\L$ is a bijection from $(-\infty,0)$ to itself, \eqref{eq-nL} follows.

\stepp Assume that $\L(\l) \in \rho(\Lp) \cap \rho(\Lm)$. Let $F = (f_1 + i f_2,g_1 + i g_2) \in L^2 \times L^2$. Let $(u_1,u_2,v_1,v_2)$ be the unique solution of \eqref{systeme} and $U = (u_1 + iu_2, v_1 + i v_2)$. Then $U \in \Dom(\Lcb)$ and $(\Lcb - \l)U = F$, so $(\Lcb-\l)$ is surjective. We already know that $\l$ is not an eigenvalue of $\Lcb$, so $(\Lcb - \l)$ is bijective. This implies that $\l$ is in the resolvent set of $\Lcb$. 

Conversely, assume that $\l \in \rho(\Lcb)$ and let $f = f_1 + i f_2 \in L^2_\C$. We denote by $u_1$ (resp. $u_2$) the first component of $(\Lob-\l)\inv (f_1,0,0,0)$ (resp. $(\Lob-\l)\inv (f_2,0,0,0)$). Then $u = u_1 + i u_2$ is such that $(\Lp-\L(\l)) u = f$, and we deduce that $\L(\l) \in \rho(\Lp)$. Similarly, $\L(\l) \in \rho(\Lm)$. This proves (i). Then (iii) follows from (i) and (ii).
\end{proof}

\begin{figure}[htpb!]
\centering
\includegraphics[width=0.40\linewidth]{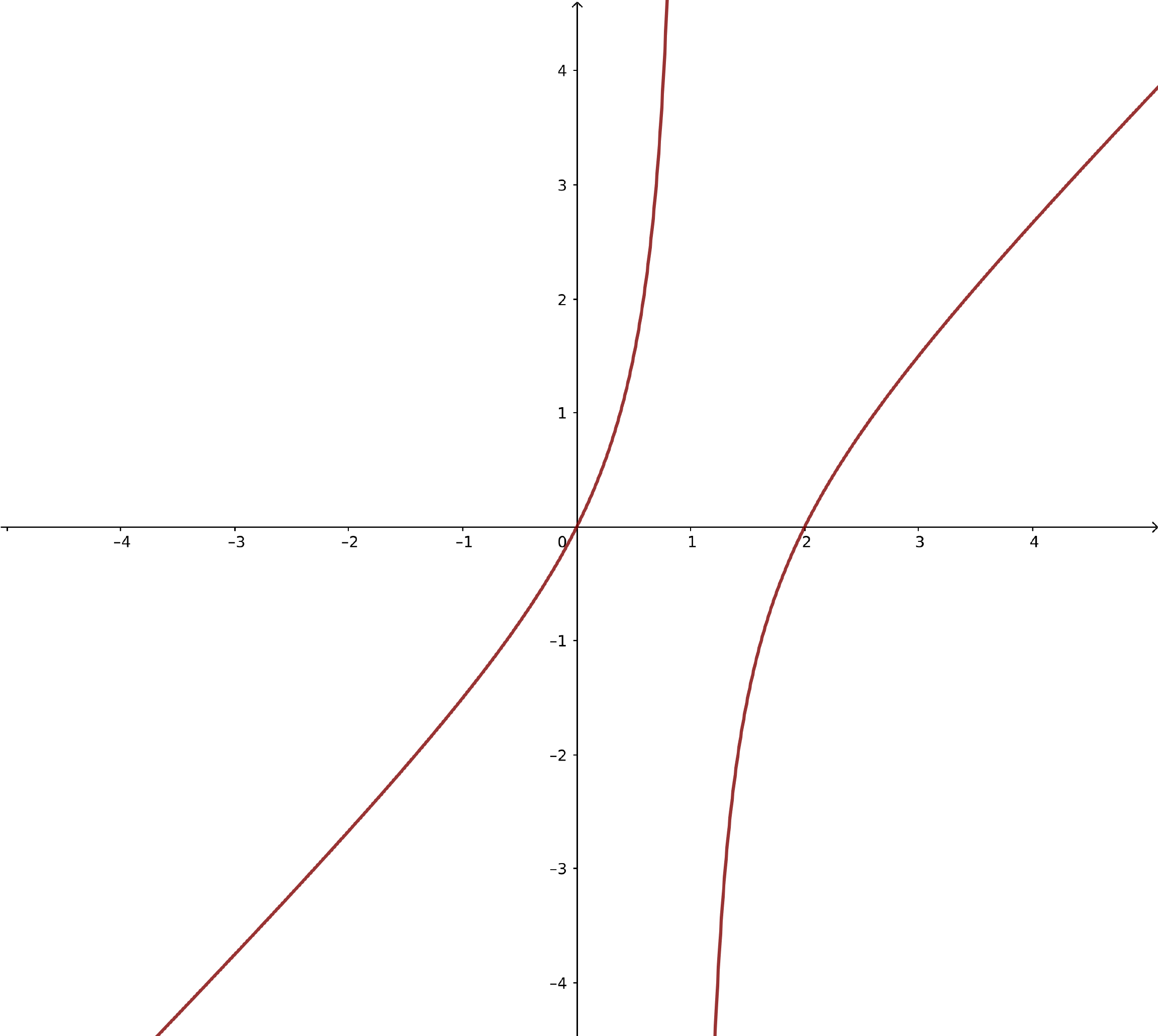}
\caption{Graph of $\lambda\mapsto\Lambda(\lambda)$}
  \label{fig:Lambda}
\end{figure}

\begin{remark} \label{biglambda}
\rm
If we denote by $\L_-$ and $\L_+$ the restrictions of $\L$ to $(-\infty,1)$ and $(1,+\infty)$, 
then $\L_- : (-\infty,1) \to \R$ and $\L_+ : (1,+\infty) \to \R$ are increasing bijections. 
Therefore,
\[
\s_\ess(\Lob) \setminus \{1\} = \L_-\inv( \s_\ess(\Lp) \cup \s_\ess(\Lm)) \cup \L_+\inv(\s_\ess(\Lp) \cup \s_\ess(\Lm)).
\]
Furthermore, since $\s_\ess(\Lp) \cup \s_\ess(\Lm)$ contains a neighborhood of $+\infty$ 
(see Proposition \ref{prop-spec-Lp-Lm} below) and $\s_\ess(\Lob)$ is closed, 
we have that $1\in\s_\ess(\Lob)$. More precisely,
\[
\s_\ess(\Lob) = [\s_1,1] \cup [\s_2,+\infty),
\]
with $\s_1 = \L_- \inv(m^2-\o^2) \in (0,1)$ and $\s_2 = \L_+ \inv(m^2-\o^2) > 1$.
\end{remark}

With Propositions \ref{prop-RLpp-Lcb} and \ref{prop-Lcb-Lp-Lm}, we can deduce 
the spectral properties of $\RLpp$ from those of $\Lp$ and $\Lm$, which we now describe.

\begin{proposition} \label{prop-spec-Lp-Lm}
Let $\b \in (-\b_0,\b_0)$.
\begin{enumerate}[\rm (i)]
\item We have $\s_\ess(\Lp) = \s_\ess(\Lm) = [m^2 - \o^2,+\infty)$.
\item The first eigenvalue of $\Lm$ is 0, it is simple, and the corresponding eigenspace 
is spanned by $\f_\b$.
\item 0 is an eigenvalue of $\Lp$ if and only if $\b = 0$. 
Moreover, $\ker(\Lpo) = \mathrm{span}(\partial_x \f_0)$ and 
the negative spectrum of $\Lpo$ reduces to a simple eigenvalue.
\end{enumerate}
\end{proposition}

\begin{proof}
It is known from \cite[Theorem I-3.1.4]{Albeverio} that the essential spectrum of 
$-\partial_{x}^2 + (m^2-\o^2)$ is $[m^2-\o^2,+\infty)$. 
As $\Lp$ and $\Lm$ are relatively compact perturbations of this operator, 
(i) follows from the Weyl Theorem (see, e.g., \cite[IV-Theorem~5.35]{Kato}).

Since $\f_\b \in D_\b$, $\Lm  \f_\b = 0$ and $ \f_\b > 0$, the first eigenvalue of $\Lm$ is $0$, it is
simple and the rest of the spectrum is positive (see, e.g., \cite[Chapter~2]{Berezin}). This proves (ii).

As for (iii), we observe that $\ffi_\b$ satisfies
\begin{equation} \label{spect4}
-\ffi''_\b + (m^2-\o^2)  \ffi_\b -  \ffi^{p}_\b = 0
\end{equation}
on $(-\infty,0)$ and on $(0,+\infty)$.
When $\b = 0$, $ \f_0$ is smooth and \eqref{spect4} holds on $\R$. 
After differentiation, we see that $ \f'_0$ belongs to $\ker(\Lpo)$. 
By Theorem 3.3 in \cite[Chapter~2]{Berezin}, 0 is a simple eigenvalue of $\Lpo$ 
and the corresponding eigenspace is spanned by $\f_0'$. 
Moreover, by \cite[Lemma 4.16]{LeCoz1}, $\Lpo$ has one simple negative eigenvalue.

Now assume that $\b \neq 0$. Let $u \in \ker(\Lp)$. In particular, $u$ satisfies
\begin{equation} \label{eq-diff-phi-beta}
- u'' + (m^2-\omega^2) u - p  \ffi^{p-1}_\b u=0
\end{equation}
on $(-\infty,0)$ and on $(0,+\infty)$. 
Since $ \f_\b'$ is also a solution of \eqref{eq-diff-phi-beta} on $(-\infty,0)$ and $(0,+\infty)$, 
there exist $\m_-,\m_+ \in \R$ such that $u = \m_-  \f_\b'$ on $(-\infty,0)$ and $u = \m_+  \f_\b'$ on $(0,+\infty)$. 
As $u$ is continuous at 0, we have 
\begin{equation} \label{eq-u0}
\m_-  \f_\b'(0^-) = u(0) = \m_+  \f_\b'(0^+).
\end{equation}
Thus $\m_- = - \m_+$ because $ \f_\b'(0^+) = -  \f_\b'(0^-) \neq 0$. Moreover, $ \f_\b$ and $u$ 
both satisfy the jump condition in \eqref{def-Db}, so 
\begin{equation} \label{eq-jump-phi-beta}
\b  \f_\b(0) =  \f_\b'(0^+) -  \f_\b'(0^-) = 2  \f_\b'(0^+)
\end{equation}
and 
\begin{equation} \label{eq-jump-u-phi}
\b u(0) = \big(u'(0^+) - u'(0^-)\big) = \m_+ \big(  \f_\b''(0^+) +  \f_\b''(0^-) \big).
\end{equation}
On the other hand, by \eqref{spect4}, 
\begin{equation} \label{eq-phi''}
 \f_\b''(0^\pm) = \lim_{x \to 0^\pm}  \f_\b''(x) = (m^2-\o^2)  \f_\b(0) -  \f_\b^p(0).
\end{equation}
Using \eqref{eq-u0}--\eqref{eq-phi''} we now have
\[
\frac {\m_+ \b^2}2  \f_\b(0) = 2 \m_+ \big( (m^2-\o^2)  \f_\b(0) -  \f_\b^p(0) \big).
\]
Since $ \f_\b(0) \neq 0$, it follows that
\[
\m_+  \f_\b^{p-1}(0) = \m_+ \left( m^2-\o^2 - \frac {\b^2} 4 \right).
\]
But a direct computation using \eqref{def-tilde-phi} gives
\[
 \f_\b^{p-1}(0) = \frac {p+1}2 \left(m^2 - \o^2 - \frac {\b^2} 4 \right).
\]
This proves that $\m_+ = 0$, whence $u = 0$. Therefore, 0 is not an eigenvalue 
of $\Lp$ when $\b \neq 0$. The proof is complete.
\end{proof}

Since $n(\Lm)=0$, \eqref{eq-nL} now gives $n(\RLpp)=n(\Lp)$.
By Proposition~\ref{prop-spec-Lp-Lm}, $\Lpo$ has a simple negative eigenvalue, 
has 0 as a simple eigenvalue, and the rest of its spectrum is positive. 
We now determine the number of negative eigenvalue of $\Lp$ for all $\b \in (-\b_0,\b_0)$ 
by a perturbation argument.

\begin{proposition}  \label{prop-n-Lp}
Let $\b \in (\b_0,\b_0)$. Then
\[
n(\Lp)=
\begin{cases}
1 & \text{if} \quad \b \leq 0, \\
2 & \text{if} \quad \b > 0.
\end{cases}
\]
\end{proposition}

\begin{proof}
We recall that $\Lp$ is associated to the form $q_\b^+$ defined in \eqref{def-qp}. 
This form is bounded from below and closed with dense domain $H^1_\C$. 
Moreover, for all $u \in H^1_\C$, the map $\b \mapsto q_\b^+(u,u)$ 
is analytic (the map $\b \mapsto \f_\b$ is pointwise analytic and locally bounded in $L^\infty$), 
so the family of operators $\Lp$ is analytic of type (B) in the sense of Kato 
(see \cite[Sec.~VII.4.2]{Kato}).
Thus, by analytic perturbation of $\Lpo$, there exist $\b_1 \in (0,\b_0)$ and an 
analytic function $\l : (-\b_1,\b_1) \to \R$ such that $\l(0) = 0$ and, for all $\b \in (-\b_1,\b_1)$, 
$\l(\b)$ is a simple eigenvalue of $\Lp$, $\Lp$ has a simple eigenvalue smaller than $\l(\b)$, 
and the rest of its spectrum is positive. Furthermore, there exists an analytic function 
$f : (-\b_1,\b_1) \to L^2_\C$ such that, for $\b \in (-\b_1,\b_1)$, $f(\b)$ belongs to $D_\C^\b$ 
and is an eigenfunction corresponding to the eigenvalue $\l(\b)$.

In particular, there exist $\l_1 \in \R$ and $f_1 \in L^2_\C$ such that 
\begin{equation} \label{Taylor1}
\l(\b) = \b \l_1 + O(\b^2),
\end{equation}
\begin{equation}\label{Taylor2}
f(\b) = \f'_0 + \b f_1 + O(\b^2).
\end{equation}
From \eqref{explicit}, $\f_\b$ is also analytic in $\b\in (-\b_0,\b_0)$ 
as a function in $H^1(\R)$, so there exists $g_1 \in H^1_\C$ such that 
\begin{equation} \label{Taylor3}
\f_\b = \f_0 + \b g_1 + O(\b^2).
\end{equation}

For $\b$ small, the sign of $\l(\b)$ is given by the sign of $\l_1$, which we now compute. 
We have
\begin{equation} \label{Taylor5}
\langle \Lp f(\beta),\ffi'_0 \rangle_{L^2_\C}
=\langle \l(\b) f(\beta),\ffi'_0 \rangle_{L^2_\C}=\lambda_1 \beta \|\ffi'_0\|^2_{L^2_\C}+O(\beta^2).
\end{equation}
On the other hand, since $\Lp$ is selfadjoint and $f(\b),\f_0' \in D_\b$, 
\begin{equation*} 
\langle \Lp f(\beta),\ffi'_0\rangle_{L^2_\C}
= \langle f(\beta),\Lp \ffi'_0 \rangle_{L^2_\C}.
\end{equation*}
Then, by \eqref{Taylor3}, 
\begin{equation*} 
\Lp \ffi'_0
= (\Lp-\Lpo)\f_0'
=-p(\ffi_\beta^{p-1}-\ffi^{p-1}_0)\ffi'_0
=-\beta p(p-1)\ffi^{p-2}_0\ffi'_0g_1+O(\beta^2).
\end{equation*}
With \eqref{Taylor2}, this yields
\begin{equation}\label{Taylor8}
\langle \Lp f(\beta), \ffi'_0\rangle_{L^2_\C}
=-\beta \langle\ffi'_0,p(p-1)\ffi^{p-2}_0\ffi'_0g_1\rangle_{L^2_\C}+O(\beta^2).
\end{equation}
A straightforward calculation using that $\Lm \f_0 = 0$ gives 
\begin{equation}\label{correct_formula}
p(p-1)\ffi_0^{p-2}(\ffi'_0)^2=\Lpo \big( (m^2-\omega^2)\ffi_0-\ffi_0^p \big).
\end{equation}
Now consider an arbitrary $\p \in H^1(\R,\R)$. 
Differentiating the identity $q_\b^-(\p,\f_\b)=0$ with respect to $\beta$ at $\beta=0$ yields
\begin{equation}\label{L_0g_1}
q_0^+( \psi,g_1) = - \psi(0)\ffi_0(0).
\end{equation}
In view of \eqref{correct_formula} and \eqref{L_0g_1}, \eqref{Taylor8} then becomes
\begin{equation*} 
\langle \Lp f(\beta),\ffi'_0\rangle_{L^2_\C}
= - \beta q_0^+ \big( (m^2-\omega^2)\ffi_0-\ffi_0^p, g_1 \big)  +O(\beta^2)
= \beta[(m^2-\omega^2)\ffi_0(0)^2-\ffi_0(0)^{p+1}]+O(\beta^2). 
\end{equation*}
Combining this with \eqref{Taylor5}, we obtain
\[
\lambda_1=\frac{(m^2-\omega^2)\ffi_0(0)^2-\ffi_0(0)^{p+1}}{\|\ffi'_0\|_{L^2_\C}^2}.
\]
But from \eqref{explicit} we have 
\[
\ffi_0(0)^{p-1}=\frac{p+1}{2}(m^2-\omega^2)>(m^2-\omega^2),
\]
hence $\lambda_1 < 0$. It follows that there exists $\b_2 \in (0,\b_1)$
such that $\Lp$ has 
exactly one negative eigenvalue for all $\b \in [-\b_2,0)$
and exactly two negative eigenvalues for all $\b \in (0,\b_2]$.

Finally, there exists $\k \in \R$ such that 
$\Lp \geq \k$ for all $\b \in (-\b_0,\b_0)$, so the negative eigenvalues of $\Lp$ are in $[\k,0)$. 
Moreover, we know from Proposition \ref{prop-spec-Lp-Lm} that 0 is not 
an eigenvalue of $\Lp$ if $\b \neq 0$. We define the contour $\Gamma$ as the boundary 
of the set $\O = (\k-1,0) + i (-1,1)$. 
Since $\G$ is in the resolvent set of $\Lp$ for $\b \neq 0$, we know from the analytic perturbation theory 
(see \cite[VII-Section~1.3]{Kato}) that the number of eigenvalues of $\Lp$ in $\O$ 
does not depend on $\b \in (-\b_0,0)$ or $\b \in (0,\b_0)$. 
Since $L^+_{\o,\b_2}$ has two negative eigenvalues and $L_{\o,-\b_2}^+$ has exactly one, 
we have $n(\Lp) = 1$ for all $\b \in (-\b_0,0)$ and $n(\Lp) = 2$ for all $\b \in (0,\b_0)$.
\end{proof}

Combining Propositions \ref{prop-RLpp-Lcb}, \ref{prop-Lcb-Lp-Lm}, 
\ref{prop-spec-Lp-Lm} and \ref{prop-n-Lp} with $\b = \g - \a \o$, we finally obtain the following result.

\begin{proposition} \label{two_negative}
Suppose $m^2-\omega^2>(\gamma-\alpha\omega)^2/4$. Then
\[
n_\om=
\begin{cases}
1 & \text{if} \quad \gamma-\alpha\omega \leq 0, \\
2 & \text{if} \quad \gamma-\alpha\omega > 0.
\end{cases}
\]
\end{proposition}

\begin{remark} \label{rad}
\rm
For $\gamma-\alpha\omega>0$, the operator $\RLpp$ 
restricted to $\HH_{\mathrm{rad}}$ has only one negative eigenvalue (see \cite[Lemma~21]{LeCoz}). 
Hence, $\frac{\diff}{\diff\o}Q(\Phi_\o)<0$ implies orbital instability in $\HH_{\mathrm{rad}}$,
and so orbital instability in $\HH$.
\end{remark}


\subsection{Slope condition}\label{slope.sec}

We shall now turn our attention to the slope condition in order to classify various stability/instability regimes.
We still consider $\alpha, \gamma$ and $\omega$ satisfying \eqref{omega_range}, and we now restore 
the dependence on $\omega$ in the notation --- which was dropped in the previous subsection, where
the parameter $\beta=\gamma-\alpha\omega$ played the key role. 

From \eqref{explicit} and \eqref{charge}, we get
\begin{align} \label{general}
Q(\Phi_\omega)
&=-\omega\|\ffi_\omega\|^2_{L^2}-\frac{\alpha}{2}|\ffi_\omega(0)|^2 \nonumber \\
&=- C(\omega)\frac{4\omega}{(p-1)\sqrt{m^2-\omega^2}}
\int^\infty_{\tau(\omega)} \text{sech}^{\frac{4}{p-1}}(y)\diff y
-\frac{\alpha}{2}C(\omega)\left(1-\frac{(\gamma-\alpha\omega)^2}{4(m^2-\omega^2)}\right)^{\frac{2}{p-1}},
\end{align}
where
\[
C(\omega)=\left(\frac{(p+1)(m^2-\omega^2)}{2}\right)^{\frac{2}{p-1}} \quad \text{and} 
\quad \tau(\omega)=\text{tanh}^{-1}\left(\frac{-(\gamma-\alpha\omega)}{2\sqrt{m^2-\omega^2}}\right).
\]
We first investigate the stability of standing waves when $p=3$, in which case \eqref{general} reduces to
\begin{align*}
Q(\Phi_\omega) 
&=2(m^2-\omega^2)
\left[\frac{-2\omega}{\sqrt{m^2-\omega^2}}
\left(1+\frac{\gamma-\alpha\omega}{2\sqrt{m^2-\omega^2}}\right)
-\frac{\alpha}{2}\left(1-\frac{(\gamma-\alpha\omega)^2}{4(m^2-\omega^2)}\right)\right] \\
&=-4 \omega\sqrt{m^2-\omega^2}-2\omega(\gamma-\alpha\omega)-\alpha(m^2-\omega^2)
+\frac{\alpha}{4}(\gamma-\alpha\omega)^2.
\end{align*}
We shall inspect the derivative of $Q(\Phi_\omega)$ with respect to $\omega$, which is given by
\begin{equation}\label{slope3}
\frac {\diff}{\diff \o} Q(\Phi_\omega)
=\frac{4\omega^2}{\sqrt{m^2-\omega^2}}-4\sqrt{m^2-\omega^2}
+\left(\frac{\alpha^3}{2}+6\alpha\right)\omega-\gamma\left(2+\frac{\alpha^2}{2}\right).
\end{equation}

In the following theorem we address the case when either $\a=0$ or $\g=0$. 
Let us first remark that in these cases there exists an $H^1$ solution of \eqref{stat}. 
Indeed, if
\[
\a=0, \quad |\g|<2m \quad \text{and} \quad \o \in (- \o_\g, \o_\g), 
\quad \text{with} \quad \o_\g = \sqrt{m^2-\frac{\g^2}{4}},
\]
or if
\[
\g=0 \quad \text{and} \quad \o \in (-\o_\a,\o_\a), \quad \text{with} \quad \o_\a = \frac{2m}{\sqrt{4+\a^2}},
\]
then the admissibility relation \eqref{omega_range} is satisfied.

\begin{theorem}\label{stability3}
Let $p=3$ and $m>0$. 
\begin{enumerate}[\rm (i)]
\item Suppose that $\alpha=0$, $|\g|< 2m$ and $\abs \o \leq \o_\g$. Let  
\[ 
\tilde \o_\g = \sqrt{ \frac {16m^2 - \g^2 + \g \sqrt{\g^2 + 32m^2}}{32}} 
= m \sqrt{\frac 12 + \frac {\g}{\sqrt{\g^2 + 32m^2} + \g}}.
\]
\begin{itemize}
\item For $\gamma<0$, $e^{i\omega t}\Phi_\omega$ is orbitally stable 
if $\abs{\o}  > \tilde \o_\g$ and orbitally unstable if $\abs \omega < \tilde \o_\g$.
\item For $\gamma > 0$, $e^{i\omega t}\Phi_\omega$ is linearly unstable if $\abs{\o} > \tilde \o_\g$ 
and orbitally unstable on $\HH_{\mathrm{rad}}$ if $\abs \omega < \tilde \o_\g$.
\end{itemize}
\item Suppose that $\gamma=0$ and $|\alpha|< 2\sqrt{\sqrt{5}-2}$. We set 
\[ 
\o_\a^\pm = \frac{m}{\sqrt{2}}\sqrt{1\mp\frac{\abs{\k}}{\sqrt{4+\k^2}}},
\quad
\text{where}
\quad 
\k = \frac 1 4 \left(\frac{\alpha^3}{2}+6\alpha\right).
\]

\medskip
\noindent
Suppose $\a < 0$.
\begin{itemize}
\item If $\o \in (-\o_\a,-\o_\a^-)$ then $e^{i\omega t}\Phi_\omega$ is orbitally stable.
\item If $\o \in (-\o_\a^-,0)$ then $e^{i\omega t}\Phi_\omega$ is orbitally unstable.
\item If $\o \in (0,\o_\a^+)$ then $e^{i\omega t}\Phi_\omega$ is orbitally unstable on 
$\HH_{\mathrm{rad}}$.
\item If $\o \in (\o_\a^+ , \o_\a)$ then $e^{i\omega t}\Phi_\omega$ is linearly unstable.
\end{itemize}

\medskip
\noindent
Suppose $\a > 0$.
\begin{itemize}
\item If $\o \in (-\o_\a,-\o_\a^+)$ then $e^{i\omega t}\Phi_\omega$ is linearly unstable.
\item If $\o \in (-\o_\a^+,0)$ then $e^{i\omega t}\Phi_\omega$ is orbitally unstable on 
$\HH_\mathrm{rad}$.
\item If $\o \in (0,\o_\a^-)$ then $e^{i\omega t}\Phi_\omega$ is orbitally unstable. 
\item If $\o \in (\o_\a^- , \o_\a)$ then $e^{i\omega t}\Phi_\omega$ is orbitally stable.
\end{itemize}
\end{enumerate}
\end{theorem}

Let us remark that orbitally instability on $\HH_{\mathrm{rad}}$ implies orbitally instability on $\HH$.

%
%

\begin{proof}
{\rm (i)} Since $\alpha=0$, we have $\frac {\diff}{\diff \o} Q(\Phi_\o) = 0$ if and only if 
\[
4 \o^2 - 2m^2 = \g \sqrt{m^2 -\o^2},
\]
that is,
\[
16 \o^4 + (\g - 16 m^2) \o^2 + (4m^4 - \g^2 m^2) = 0 \quad \text{and} \quad \mathrm{sgn}(2\o^2 - m^2) 
= \mathrm{sgn}(\g).
\]
The only possibility is $\o^2 = \tilde \o_\g^2$. Since $\frac {\diff}{\diff \o} Q(\Phi_\o)$ is negative when $\o = 0$ 
and goes to $+\infty$ when $\o$ goes to $\pm \m$, we deduce that 
$\frac {\diff}{\diff \o} Q(\Phi_\o) < 0$ if and only if $\abs \o \leq \tilde \o_\g$.
Furthermore, Proposition~\ref{two_negative} with $\alpha=0$ yields $n_\o = 1$ if $\gamma<0$
and $n_\o = 2$ if $\gamma>0$.
Hence, the conclusions in (i) follow from Proposition~\ref{gss1}, Proposition~\ref{gss2} and Remark~\ref{rad}. 

{\rm (ii)} We now consider the case $\g = 0$. 
We have that $\frac {\diff}{\diff \o} Q(\Phi_\o) = 0$ if and only if 
\[
\o^4 - m^2 \o^2 + \frac {m^4}{4+\k^2} = 0 \quad \text{and} \quad \mathrm{sgn}(2\o^2 - m^2) 
= \mathrm{sgn}(-\k \o).
\]
The solutions are $-\o_\a^-$ and $\o_\a^+$. Then, for $\abs \o < \o_\a$, 
we have $\frac {\diff}{\diff \o} Q(\Phi_\o) < 0$ if and only if $\o \in (-\o_\a^-,\o_\a^+)$. Furthermore, 
$n_\o = 1$ if $\a\o \geq 0$ and $n_\o = 2$ if $\a\o < 0$. 
Hence, the conclusions again follow from Proposition~\ref{gss1}, 
Proposition~\ref{gss2} and Remark~\ref{rad}. 
\end{proof}

\begin{remark}
\rm
Notice that $\tilde \o_\g > \o_\g$ when $\g >  2m/\sqrt 3$. 
In this case, $e^{i\o t} \Phi_\o$ is orbitally unstable for all $\o \in (-\o_\g,\o_\g)$.

Similarly, if $\abs\a \geq 2\sqrt{\sqrt{5}-2}$ then $\o_\a^+ \geq \o_\a$, 
so the set of $\o$ for which we have linear instability is empty.
\end{remark}


We next give some results with non-zero coupling constants, 
$\gamma\neq0$ and $\alpha\neq0$. We first observe that the right-hand side 
of \eqref{slope3} vanishes for
\[
\gamma=\tilde{\gamma}(\alpha,\omega):=
\frac{2}{4+\alpha^2}\left[\frac{4\omega^2}{\sqrt{m^2-\omega^2}}-4\sqrt{m^2-\omega^2}
+\left(\frac{\alpha^3}{2}+6\alpha\right)\omega\right].
\]
It follows that 
\begin{align*}
\sgn\frac {\diff}{\diff \o} Q(\Phi_\omega)
&=-\sgn(\gamma-\tilde{\gamma}).
\end{align*}
The following theorem is then proved
using Proposition~\ref{gss1}, Proposition~\ref{gss2} and Remark~\ref{rad}, similarly to the proof
of Theorem~\ref{stability3}.

\begin{theorem}
Let $p=3$ and consider $\omega\in (-m,m)$, $\a\in\R$ and $\g\in\R$ satisfying \eqref{omega_range}.
\begin{itemize}
\item[(i)] Suppose $\gamma-\alpha\omega<0$. 
Then $e^{i\omega t}\Phi_\omega$ is orbitally stable if $\gamma<\tilde{\gamma}$
and orbitally unstable if $\gamma>\tilde{\gamma}$.
\item[(ii)] Suppose $\gamma-\alpha\omega>0$. 
Then $e^{i\omega t}\Phi_\omega$ is linearly unstable if $\gamma<\tilde{\gamma}$
and orbitally unstable in $\HH_{\mathrm{rad}}$ if $\gamma>\tilde{\gamma}$.
\end{itemize}
\end{theorem}

\begin{remark}
\rm
For any fixed $\alpha\in\R$, there always exist values of the parameters $\o, \g$ 
satisfying the above conditions for stability/instability.
For instance, conditions \eqref{omega_range}, $\g<\a\o $ and $\g < \tilde{\g}$ are all satisfied provided
\[
0<\tilde{\g}-\a\o+2\sqrt{m^2-\o^2}=
\frac{8}{4+\alpha^2}\frac{2\omega^2-m^2}{\sqrt{m^2-\omega^2}}+2\sqrt{m^2-\omega^2}
+\frac{8\alpha\omega}{4+\alpha^2}.
\]
Clearly, this condition is satisfied for $|\o| \approx m$. 
The other cases follow by similar arguments.
\end{remark}


We now consider more general values of the power $1<p<5$. To keep the exposition
simple enough, we focus on the cases where the coupling constants $\a,\g$ have the same
sign. Of course, mixed cases could also be considered.

\begin{lemma} \label{slope_general}
Let $1<p<5$.
\begin{enumerate}[\rm (i)]
\item Suppose $\alpha, \gamma>0$.
We have
\[
\frac {\diff}{\diff \o} Q(\Phi_\omega)<0 \quad \text{for} \ \omega\in\left(-\frac{m}{2}\sqrt{p-1},0\right)
\]
and
\[
\frac {\diff}{\diff \o} Q(\Phi_\omega)>0 \quad \text{for} \ 
\omega\in\left(\frac{\alpha\gamma}{4+\a^2},\frac{\alpha m^2}{\gamma}\right)
\cap \left(\frac{m}{2}\sqrt{p-1},m\right),
\]
whenever these intervals are not empty.
\item Suppose $\alpha, \gamma<0$. 
We have
\[
\frac {\diff}{\diff \o} Q(\Phi_\omega)<0 \quad \text{for} \ 
\omega\in\left(0,\frac{m}{2}\sqrt{p-1}\right)
\cap\left(\frac{\alpha\gamma}{4+\a^2},\frac{\alpha m^2}{\gamma}\right)
\]
and
\[
\frac {\diff}{\diff \o} Q(\Phi_\omega)>0 \quad \text{for} \ 
\o \in\left ( \frac{\a m^2}{\g}, \frac{\a\g}{4+\a^2} \right) \cap \left( \frac{m}{2}\sqrt{p-1}, m \right),
\]
whenever these intervals are not empty.
\end{enumerate}
\end{lemma}

\begin{proof}
We only prove (i), as (ii) is proved by similar calculations. 
We rewrite \eqref{general} as
\[
Q(\Phi_\omega)=C_1(\omega)I(\omega)+C_2(\omega),
\]
where
\begin{align*}
C_1(\omega)
&=-\frac{4}{p-1}\left( \frac{p+1}{2}\right)^{\frac{2}{p-1}}
\omega(m^2-\omega^2)^{\frac{2}{p-1}-\frac{1}{2}}, \\
I(\omega)
&=\int^\infty_{\tau(\omega)}\mathrm{sech}^{\frac{4}{p-1}}(y)\diff y, \\
C_2(\omega)
&=-\frac{\alpha}{2}\left(\frac{p+1}{8} \right)^{\frac{2}{p-1}}(4(m^2-\omega^2)
-(\gamma-\alpha\omega)^2)^{\frac{2}{p-1}}.
\end{align*}
We first find 
\[
\frac{\partial C_2}{\partial \omega} = 
-\frac{\alpha}{p-1}\left(\frac{p+1}{8} \right)^{\frac{2}{p-1}}(4(m^2-\omega^2)
-(\gamma-\alpha\omega)^2))^{\frac{2}{p-1}-1}(-8\omega+2\alpha(\gamma-\alpha\omega)).
\]
It follows that $\frac{\partial C_2}{\partial\omega} >0$ if
\begin{equation} \label{range1}
\omega>\frac{\alpha\gamma}{4+\a^2}.
\end{equation}
We next determine the sign of $\frac{\partial}{\partial \o} (C_1I)$ assuming $\omega>0$.
Since $I$ is positive and $C_1$ is negative for $\o>0$, 
we will have 
\[
\frac{\partial }{\partial\omega}(C_1I)
=\frac{\partial C_1}{\partial\omega}I+C_1\frac{\partial I}{\partial\omega}>0,
\]
provided $\frac{\partial C_1}{\partial\omega}>0$ and $\frac{\partial I}{\partial\omega}<0$.
On the one hand, we have
\[
\frac{\partial C_1}{\partial\omega}=
-\frac{4}{p-1}\left(\frac{p+1}{2}\right)^{\frac{2}{p-1}}(m^2-\omega^2)^{\frac{2}{p-1}-\frac{1}{2}}
+\frac{8}{p-1}\left(\frac{p+1}{2}\right)^{\frac{2}{p-1}}
\left( \frac{2}{p-1}-\frac{1}{2}\right)\omega^2(m^2-\omega^2)^{\frac{2}{p-1}-\frac{3}{2}},
\]
which is positive if
\begin{equation} \label{range2}
|\omega|>\frac{m}{2}\sqrt{p-1}.
\end{equation}
On the other,
\[
\frac{\partial I}{\partial\omega}
=-\mathrm{sech}^{\frac{4}{p-1}}(\tau)\frac{\partial}{\partial\omega}\tau
=-\frac{1}{\sqrt {m^2-\omega^2}}\left(1-
\frac{(\gamma-\alpha\omega)^2}{4(m^2-\omega^2)}\right)^{\frac{2}{p-1}}
\frac{2\alpha m^2-2\gamma\omega}{4(m^2-\omega^2)-(\gamma-\alpha\omega)^2},
\] 
which is negative if
\begin{equation} \label{range3}
\omega<\frac{\alpha m^2}{\gamma}.
\end{equation}
Hence, it follows from \eqref{range1}--\eqref{range3} 
that $\frac {\diff}{\diff \o}Q(\Phi_\omega) >0$ if
\[
\omega\in \left(\frac{\alpha\gamma}{4+\a^2},
\frac{\alpha m^2}{\gamma}\right)\cap \left(\frac{m}{2}\sqrt{p-1},m\right).
\]

We now show that $\frac {\diff}{\diff \o} Q(\Phi_\omega)<0$ 
for $\o\in\left(-\frac{m}{2}\sqrt{p-1},0\right)$.\\ 
As $C_1$ and $I$ are both positive for $\o\in(-m,0)$, we will have
\[
\frac{\partial }{\partial\omega}(C_1I)
=\frac{\partial C_1}{\partial\omega}I+C_1\frac{\partial I}{\partial\omega}<0,
\]
provided $\frac{\partial C_1}{\partial\omega}<0$ and 
$\frac{\partial I}{\partial\omega}<0$. 
From the previous calculations, we know that $\frac{\partial C_1(\omega)}{\partial\omega}<0$ if
\begin{equation} \label{range4}
|\omega|<\frac{m}{2}\sqrt{p-1}
\end{equation}
and $\frac{\partial I}{\partial\omega}<0$ if
\begin{equation} \label{range5}
\omega<\frac{\alpha m^2}{\gamma}.
\end{equation}
Finally, $\frac{\partial C_2}{\partial\omega}<0$ if
\begin{equation} \label{range6}
\omega<\frac{\alpha\gamma}{4+\a^2}.
\end{equation}
Since $\alpha>0$ and $\gamma>0$, we conclude from \eqref{range4}--\eqref{range6}  
that $\frac {\diff}{\diff \o} Q(\Phi_\omega)<0$ for all $\omega\in\left(-\frac{m}{2}\sqrt{p-1},0\right)$.
\end{proof}

We finally combine Lemma~\ref{slope_general} with the
spectral conditions in Proposition~\ref{two_negative} to get the following result.

\begin{theorem} \label{stab_general}
Let $1<p<5$ and $m=1$.
\begin{enumerate}[\rm (i)]

\item Let $\alpha, \gamma>0$.
If $\frac{\a\g}{4+\a^2}<1$ and 
$1<\frac{\g^2}{4+\a^2}+\frac{2\g}{\a(4+\a^2)}\sqrt{4+\a^2-\g^2}$, 
then there exists $\o$ satisfying \eqref{omega_range} and
\[
\omega\in \left(\frac{\alpha\gamma}{4+\a^2},\frac{\alpha}{\gamma} \right)
\cap \left(\frac{1}{2}\sqrt{p-1},1\right).
\]
For such $\omega$, the standing wave $e^{i\omega t}\Phi_{\omega}$ is orbitally stable.
\\
If $\a\g<2\sqrt{4+\a^2-\g^2}$, then there exists $\o$ satisfying 
\eqref{omega_range} and
\[
\omega \in \left(-\frac{1}{2}\sqrt{p-1},0 \right).
\]
For such $\omega$, $e^{i\omega t}\Phi_{\omega}$ is orbitally unstable.

\item Let $\alpha, \gamma<0$.
If $1<\frac{\g}{\a}$, then there exists $\o\in\R$ satisfying 
\eqref{omega_range} and
\[
\omega\in \left( \frac{\alpha}{\gamma}, \frac{\alpha\gamma}{4+\a^2}\right)
\cap \left(\frac{1}{2}\sqrt{p-1},1\right).
\]
For such $\omega$, $e^{i\omega t}\Phi_{\omega}$ is orbitally stable.
\\
If $\frac{\g}{\a}<\frac{1}{2}\sqrt{p-1}$, then there exists $\o\in\R$ satisfying 
\eqref{omega_range} and 
\[
\omega\in\left(0,\frac{1}{2}\sqrt{p-1}\right)
\cap\left(\frac{\alpha\gamma}{4+\a^2},\frac{\alpha}{\gamma}\right).
\]
For such $\omega$, $e^{i\omega t}\Phi_{\omega}$ is orbitally unstable.
\end{enumerate}
\end{theorem}

\begin{proof}
We only prove (i), the proof of (ii) being similar. The hypotheses $m=1$ and 
$\g<\a$ imply that there exists $\o\in\R$ satisfying
\[
\omega\in \left(\frac{\alpha\gamma}{4+\a^2},\frac{\alpha}{\gamma}\right)
\cap \left(\frac{1}{2}\sqrt{p-1},1\right).
\]
In particular, since $\frac{\alpha\gamma}{4+\a^2}<\frac{\g}{\a}$, we also have $\o<\frac{\g}{\a}$.
Furthermore, 
if $1<\frac{\g^2}{4+\a^2}+\frac{2\g}{\a(4+\a^2)}\sqrt{4+\a^2-\g^2}$, then $\o$ 
satisfies the admissibility condition \eqref{omega_range}. 
Orbital stability then follows from Proposition~\ref{two_negative} and Lemma~\ref{slope_general}.

The condition $\a\g<2\sqrt{4+\a^2-\g^2}$ implies that 
there exists a $\omega \in \left(-\frac{1}{2}\sqrt{p-1},0 \right)$ satisfying \eqref{omega_range}. 
Orbital instability follows from Proposition~\ref{two_negative}, Remark~\ref{rad} and Lemma~\ref{slope_general}.
\end{proof}

\begin{remark}
\rm
The condition $1<\frac{\g^2}{4+\a^2}+\frac{2\g}{\a(4+\a^2)}\sqrt{4+\a^2-\g^2}$ is satisfied if 
$0<\a<\g<\sqrt{4+\a^2}$. 
\end{remark}

{\bf Acknowledgements} The authors are indebted to Stefan Le Coz for many fruitful discussions about this paper. This research was supported by CIMI Labex, Toulouse, France, under grant ANR-11-LABX-0040-CIMI, by the French ministries of Europe and foreign affairs (MEAE), of higher education, research and innovation (MESRI) via the partnership Hubert Curien (PHC) Van Gogh 2018 37915YF, and by the Van Gogh travel grant 2018 13077.

\end{document}